\documentclass[11pt,a4paper]{amsart}
\usepackage{amsmath}
\usepackage{amsfonts}
\usepackage{graphicx}
\usepackage{framed}
\usepackage{amssymb,cancel}
\usepackage{xcolor}
\usepackage{esint}

\usepackage{etex}
\usepackage{latexsym}
\usepackage[english]{babel}
\usepackage[utf8x]{inputenc}

\def \R {\mathbb{R}}

\def \de {\partial}

\def \e {\varepsilon}

\def \LL {\mathcal{L}_{\varepsilon}}

\def \Oo {\mathcal{O}}
\DeclareMathOperator*{\esssup}{ess\,sup}
\DeclareMathOperator*{\essinf}{ess\,inf}

\usepackage{amsthm}
\theoremstyle{definition}
\newtheorem{definition}{Definition}[section]

\newtheorem{remark}[definition]{Remark}

\theoremstyle{plain}
\newtheorem{theorem}[definition]{Theorem}
\newtheorem{proposition}[definition]{Proposition}
\newtheorem{lemma}[definition]{Lemma}
\newtheorem{corollary}[definition]{Corollary}

\numberwithin{equation}{section}
\usepackage[colorlinks=true,urlcolor=blue,
citecolor=red,linkcolor=blue,linktocpage,pdfpagelabels,
bookmarksnumbered,bookmarksopen]{hyperref}
\usepackage[hyperpageref]{backref}

\makeatletter
\@namedef{subjclassname@2020}{\textup{2020} Mathematics Subject Classification}
\makeatother

\begin{document}

 \title[Mixed local-nonlocal singular critical problems]
 {Multiplicity of positive solutions for mixed local-nonlocal singular critical problems}

 \author[S.\,Biagi]{Stefano Biagi}
 \author[E.\,Vecchi]{Eugenio Vecchi}
 
 \address[S.\,Biagi]{Dipartimento di Matematica
 \newline\indent Politecnico di Milano \newline\indent
 Via Bonardi 9, 20133 Milano, Italy}
 \email{stefano.biagi@polimi.it}
 
 \address[E.\,Vecchi]{Dipartimento di Matematica
 \newline\indent Università di Bologna \newline\indent
 Piazza di Porta San Donato 5, 40126 Bologna , Italy}
 \email{eugenio.vecchi2@unibo.it}

\subjclass[2020]{35J75, 35M12, 35B33}

\keywords{Mixed local-nonlocal operators, singular PDEs, critical problems}

\thanks{The Authors are
member of the {\em Gruppo Nazionale per
l'Analisi Ma\-te\-ma\-ti\-ca, la Probabilit\`a e le loro Applicazioni}
(GNAMPA) of the {\em Istituto Nazionale di Alta Matematica} (INdAM), and are
partially 
supported by the PRIN 2022 project 2022R537CS \emph{$NO^3$ - Nodal Optimization, NOnlinear elliptic equations, NOnlocal geometric problems, with a focus on regularity}, founded by the European Union - Next Generation EU}

 \begin{abstract}
 We prove the existence of at least two positive weak solutions for mixed local-nonlocal singular and critical semilinear elliptic problems in the spirit of \cite{Haitao}, extending the recent results in \cite{Garain} concerning singular problems and, at the same time, the results in \cite{BDVV5} regarding critical problems.
 \end{abstract}
 \maketitle 
\section{Introduction}\label{sec.Intro}
Let $\Omega \subset \mathbb{R}^n$ be an open and bounded set with smooth enough boundary $\partial \Omega$. We consider the following mixed local-nonlocal singular and critical semilinear elliptic problem:
\begin{equation}\tag{{$\mathrm{P_{\e}}$}}\label{eq:Problem}
\left\{ \begin{array}{rl}
\LL u  = \dfrac{\lambda}{u^{\gamma}} + u^{2^{\ast}-1} & \textrm{in } \Omega,\\
u>0 & \textrm{in } \Omega,\\
u= 0 & \textrm{in } \mathbb{R}^{n}\setminus \Omega,
\end{array}\right.
\end{equation}
 \noindent where $\lambda >0$ is a positive real parameter, $\gamma \in (0,1)$ and
 $$\LL u = -\Delta u + \e(-\Delta)^s u, \qquad s\in (0,1), \quad \e \in (0,1].$$ 
 The choice of $1$ as upper bound for $\e$ is purely esthetic.
  Along the paper it will sometimes be useful to denote the above problem as \eqref{eq:Problem}$_\lambda$ 
   to make it clear the choice of the parameter. 
\medskip

The study of singular problems when $\e=0$, i.e. $\mathcal{L}_{0}=-\Delta$, is pretty old and it dates back to \cite{Fulks} and 
\cite{Stuart}; nevertheless, it is probably the famous paper by Crandall, Rabinowitz and Tartar \cite{CRT} which 
has raised the attention to such a kind of singular PDEs. This attention
in\-cre\-a\-sed 
even more after the celebrated paper by Lazer and Mckenna \cite{LazMck}, where it was proved that for 
{\it very} singular problems (i.e., $\gamma \geq 3$), classical solutions may not belong to the natural Sobolev 
space. This is one of the reasons why variational methods are not always fruitful in dealing with singular 
problems, leading to approximation techniques extensively developed after \cite{BO} and used also in presence of 
measure sources \cite{OP}. Usually, most of the difficulties appear in cases when the right hand side takes the 
form 
$$\dfrac{f(x)}{u^{\gamma}},$$
\noindent where $f$ may be only a $L^1$ function. The situation is obviously easier for $f$ constant plus a possible perturbation (that may or may not be of power type). In this case, the moving plane has been settled in \cite{CGS} and variational methods may be used, see e.g. \cite{CoPa, Hirano, Haitao}. Since we will deal with mixed local-nonlocal problems, we also want to mention a few purely nonlocal results dealing with singular problems, see \cite{BDBMP, CMSS, CMST}, dealing with existence and qualitative properties of solutions.
\medskip

As already mentioned, the main character involved in this paper is the linear mixed local-nonlocal operator,
$$\mathcal{L}_{\e}:= -\Delta + \e(-\Delta)^s, \quad s \in (0,1), \quad \e \in (0,1],$$
\noindent see Section \ref{sec.Prel}.
The operator $\mathcal{L}_\e$ written above is just a special instance of a wide more general class of operators whose study began in the '60s, see \cite{BCP} and \cite{Cancelier} for generalizations, in connection with the validity of a maximum principle. On the other hand, the operator $\mathcal{L}$ can be seen as the infinitesimal generator of a stochastic process having both a Brownian motion and a L\'{e}vy flight, and hence there is a vast literature which establishes several regularity properties 
adopting probabilistic techniques, see e.g. \cite{CKSV} and the references therein. 

More recently, the study of regularity properties related to this operator (and its quasilinear generalizations) has seen an increasing interest, mainly adopting more analytical and PDEs approaches, see, e.g., 
\cite{AC, BDVV, BMS, CDV22, DeFMin, DPLV2, GarainKinnunen, GarainLindgren, SVWZ}. 
It is worth mentioning that the operator $\mathcal{L}_\e$ seems to be
of interest in biological applications, see, e.g.,
\cite{DV} and the references therein.
\medskip

The second actor, appearing in the right hand side of \eqref{eq:Problem}, is given by the critical perturbation of a mild singular term, namely
$$\dfrac{\lambda}{u^{\gamma}} + u^{2^{\ast}-1},$$
\noindent where $\lambda >0$, $\gamma \in (0,1)$ (hence the term {\it mild}) and $2^{\ast} := \tfrac{2n}{n-2}$ (for $n \geq 3$) is the critical Sobolev exponent. The problem we are interested in is pretty classical and can be formulated as follows
\begin{center}
{\it Find values of the parameter $\lambda$ for which \eqref{eq:Problem} admits one or more positive weak solutions.}
\end{center}
We stress that, for every $\e \in (0,1]$, one is dealing with a problem
which depends on $\e$; hence,
any possible solution $u_{\lambda}$ will {\it depend on $\e$} in a not explicit way. 
\vspace*{0.1cm}

Before entering in the details of our result, we first make a
 brief review of the existing literature concerning mixed local-nonlocal singular problems (with $\e =1$).
\begin{itemize}
\item In \cite{ArRa}, the authors prove existence and regularity results of positive solutions 
of the following \emph{linear and purely singular} problem 
\begin{equation} \label{eq:unperturbedIntro}
 \begin{cases}
  -\Delta u + (-\Delta)^s u = \dfrac{f(x)}{u^{\gamma}} & \text{in $\Omega$}, \\
  u = 0 & \text{in $\R^n\setminus\Omega$},
  \end{cases}
\end{equation}
where $\gamma > 0$ is a fixed parameter.
\vspace*{0.1cm}

\item In \cite{GarainUkhlov}, the authors prove existence of positive solutions of
the \emph{quasilinear counterpart} of the singular problem \eqref{eq:unperturbedIntro}, that is,
$$
\begin{cases}
 -\Delta_p u + (-\Delta_p)^s u = \dfrac{f(x)}{u^{\gamma}} & \text{in $\Omega$}, \\
  u = 0 & \text{in $\R^n\setminus\Omega$},
  \end{cases}
$$
 in connection with mixed local-nonlocal Sobolev inequalities. 
Regularity of solutions un\-der different summability assumptions on $f$ and symmetry results are proved as well.
\vspace*{0.1cm}

\item In \cite{Garain}, Garain proved the existence of at least two positive solutions for a 
\emph{subcritical} perturbation of problem \eqref{eq:unperturbedIntro}, namely
$$
\begin{cases}
 -\Delta_p u + (-\Delta_p)^s u = \dfrac{\lambda}{u^{\gamma}} + u^p & \text{in $\Omega$}, \\
  u = 0 & \text{in $\R^n\setminus\Omega$},
  \end{cases}
$$
where $\lambda >0,\,\gamma \in (0,1)$ and $p \in (1, 2^{\ast}-1)$.
The case of nonlinearities of the form $g(x,u) = \lambda h(u)u^{-\gamma}$ is also considered.
\vspace*{0.1cm}


\item In \cite{Biroud}, Biroud extended some of the results in \cite{GarainUkhlov} to the case of variable exponent $\gamma(x)$.
\vspace*{0.1cm}

\item In \cite{GKK}, the authors proved some existence results
for a mixed anisotropic operator with variable singular exponent.
\end{itemize}

\medskip
Keeping the previously mentioned results in mind, we are now ready to state the first main result of this paper. In what follows,
we refer to Definition \ref{def:weaksol} for the precise definition of \emph{weak solution}
 of \eqref{eq:Problem}$_{\lambda}$.
\begin{theorem} \label{thm:main}
 Let $\Omega\subset\R^n$ \emph{(}with $n\geq 3$\emph{)} be a bounded open set
 with smooth enough boundary, and let $\gamma\in (0,1)$. Then, for every $\e \in (0,1]$, there exists $\Lambda_{\e} > 0$
 such that
 \begin{itemize}
  \item[a)]  problem \eqref{eq:Problem}$_{\lambda}$ admits at least one weak solution for every $0<\lambda\leq \Lambda_{\e}$;
  \item[b)]
    problem \eqref{eq:Problem}$_{\lambda}$ does not admit weak solutions
  for every $\lambda>\Lambda_{\e}$.
 \end{itemize}
 Moreover, for every $\lambda \in (0,\Lambda_{\e})$, the weak solution $u_{\lambda}$ found above is a local minimizer of the functional
\begin{equation}\label{eq:Def_Ilambda}
I_{\lambda,\e}(u):= \dfrac{1}{2}\rho_{\e}(u)^2 - \dfrac{\lambda}{1-\gamma}\int_{\Omega}|u|^{1-\gamma}\, dx - \dfrac{1}{2^*}\int_{\Omega}|u|^{2^*}\, dx
\end{equation}
in the $\mathcal{X}^{1,2}(\Omega)$-topology.
\end{theorem}

The above theorem falls in the framework of the results of Haitao \cite{Haitao} and Hirano, Saccon and Shoji \cite{Hirano}, where the authors considered critical perturbations of mild singular terms using sub and supersolution methods (\cite{Haitao}) or a Nehari manifold approach (\cite{Hirano}). We also note that Theorem \ref{thm:main} extends the results in \cite{Garain}.
\medskip

Let us now spend a few words on the proof of Theorem \ref{thm:main}. First of all, we decided to follow the approach of Haitao \cite{Haitao}. The argument runs as follows. The existence of a first positive solution is obtained by means of a sub/su\-per\-so\-lu\-tion variational scheme, which extends to the critical case a result in \cite{Garain}. The role of the subsolution is naturally played by the solution of the purely singular problem \eqref{eq:unperturbedIntro}, 
while a supersolution can be found exploiting the monotonicity of $\eqref{eq:Problem}_{\lambda}$ with respect to the parameter, see Lemma \ref{lem:2.3Haitao}.
As in \cite{Haitao}, in Lemma \ref{lem:25Haitao}, the first solution found is proved to be a local minimizer in the $\mathcal{X}^{1,2}(\Omega)$-topology for the functional $I_{\lambda, \e}$ defined in \eqref{eq:Def_Ilambda}. In the purely local case \cite{Haitao}, 
this is proved using a quantitative strong maximum principle established 
in \cite[Theorem 3]{BNH1C1}. Since we do not have such a result at our disposal, we exploit a weak Harnack-type inequality for weak supersolutions of mixed operators with a zero order term, see Section \ref{sec:weakharnack}. 
In doing this, we heavily rely on the recent regularity results obtained by Garain and Kinnunen \cite{GarainKinnunen}.
The fact that the first solution is a local minimizer is the key point to start when looking for a second positive solution. Classically, this is done by Ekeland variational principle combined with a perturbation of the first solution with a suitable family of functions built upon the Aubin-Talenti functions. This choice is pretty natural in the purely local case and it is replaced by the nonlocal Talenti functions in the purely nonlocal case \cite{GiacMukSre}. Thanks to the recent results in \cite{BDVV5}, in the mixed local-nonlocal case the situation is closer to the purely local case: indeed, in \cite[Propositions 1.1 and 1.2]{BDVV5} it has been proved that the best constant in the natural mixed Sobolev quotient is never achieved and that it coincides with the purely local one, therefore suggesting that the classical Aubin-Talenti functions are the good candidates to play with in critical problems. Despite this, the presence of the fractional Laplacian creates a lack of scaling invariance of the operator $\mathcal{L}_{\e}$ (and this occurs for every $\e$) which prevents from a clear repetition of the classical arguments, see the proof of Lemma \ref{lem:CrucialLemma}. See also \cite{KRS} where a similar situation occurs with the $(p,q)$-Laplacian. 
\vspace*{0.1cm}

To deal with this phenomenon we assume that $\e$ is \emph{sufficiently small}
 (so that the role of the fractional Laplacian $(-\Delta)^s$ in the lacking of homogeneity
 is somehow negligible), and we have to carefully keep track of the dependence on $\e$ of several crucial estimates. 
 Once done this, we can  prove the existence of a second solution.
 \begin{theorem} \label{thm:main2}
 Let $\Omega\subset\R^n$ \emph{(}with $n\geq 3$\emph{)} be a bounded open set
 with smooth enough boundary, and let $\gamma\in (0,1)$. 
 Then, there exist 
 $\lambda_{\star} > 0$ (independent of $\e$) and $\varepsilon_0 \in (0,1]$
 such that problem \eqref{eq:Problem}$_{\lambda}$
admits \emph{two weak solutions} for all $\varepsilon \in (0,\varepsilon_0)$ and all $\lambda \in (0,\lambda_{\star})$. 
\end{theorem}


We notice that a careful inspection of Lemma \ref{lem:CrucialLemma} shows that a better result can be obtained at the price of quite big restrictions, namely
\begin{corollary}
Let $n=3$ and $s\in \left(0,\tfrac{1}{2}\right)$. Then, for every $\e \in (0,1]$ problem \eqref{eq:Problem}$_{\lambda}$ admits at least two positive weak solutions for every $\lambda \in (0,\Lambda)$.
\end{corollary}
\medskip

The paper is organized as follows: in Section \ref{sec.Prel} we collect all the relevant notation, definitions and preliminary results; in Section \ref{sec:weakharnack} we prove a weak Harnak-type inequality for weak supersolutions of $\mathcal{L}_{\e} + a$ which we believe being of independent interest. In Section \ref{sec:proofThm} we prove Theorem \ref{thm:main} following the scheme previously described and finally in Section \ref{sec:proofThm2} we prove Theorem \ref{thm:main2}.

\section{Preliminaries}\label{sec.Prel}
 In this section we collect some preliminary definitions and results which
 will be used in the rest of the paper. First of all, we review some basic properties
 of the fractional Laplace operator $(-\Delta)^s$; 
 we then properly
 introduce the adequate functional setting for the study
 of problem \eqref{eq:Problem}$_{\lambda}$, and we give the precise definition
 of \emph{weak sub/supersolution} of this problem.
 \medskip
 
 \noindent\textbf{i) The fractional Laplacian.} Let $s\in (0,1)$ be fixed, and let
$u:\R^n\to\R$ be a measurable function. The \emph{fractional La\-pla\-cian} (of order $s$) of $u$
at a point $x\in\R^n$ is defined (up to a
\emph{normalization constant} $C_{n,s} > 0$ that we neglect) as
\begin{equation*}
 (-\Delta)^s u(x)  = 2\,\mathrm{P.V.}\int_{\R^n}\frac{u(x)-u(y)}{|x-y|^{n+2s}}\,dy
 = 2\,\lim_{\varepsilon\to 0^+}\int_{\{|x-y|\geq\varepsilon\}}\frac{u(x)-u(y)}{|x-y|^{n+2s}}\,dy,
\end{equation*}
provided that the above limit \emph{exists and is finite}.

Since we are interested in \emph{weak solutions}
of problem \eqref{eq:Problem}$_{\lambda}$, 
we recall some basic
facts concerning the \emph{Weak Theory} of $(-\Delta)^s$.
To begin with we recall that, if $\mathcal{O}\subseteq\R^n$ is an arbitrary open set,
the {fractional Laplacian} $(-\Delta)^s$ is related
(essentially via the Euler-Lagrange equation)
to the \emph{fractional Sobolev space} $H^s(\Oo)$, which is defined as follows:
$$H^s(\Oo) := \Big\{u\in L^2(\Oo):\,[u]^2_{s,\Oo} = \iint_{\Oo\times\Oo}
\frac{|u(x)-u(y)|^2}{|x-y|^{N+2s}}\,dx\,dy < \infty\Big\}.$$
While we refer to \cite{LeoniFract} for a thorough introduction
on fractional Sobolev spaces, here we list the few basic properties of $H^s(\Oo)$
we will exploit in this paper.
\begin{itemize}
 \item[a)] $H^s(\Oo)$ is a real Hilbert space, with the scalar product
 $$\langle u,v\rangle_{s,\,\Oo} 
 := \iint_{\Oo\times\Oo}\frac{(u(x)-u(y))(v(x)-v(y))}{|x-y|^{N+2s}}\,dx\,dy\qquad
 (u,v\in H^s(\Oo)).$$ 
 
 \item[b)] $C_0^\infty(\Oo)$ is a \emph{linear subspace of $H^s(\Oo)$}; in addition,
 in the particular case when $\Oo = \R^n$, we have that
  $C_0^\infty(\R^n)$ is \emph{dense} in $H^s(\R^n)$.
 \vspace*{0.1cm}
 
 \item[c)] If $\Oo = \R^n$ or if $\Oo$ has \emph{bounded boundary $\partial\Oo\in C^{0,1}$},
 we have the \emph{continuous embedding} $H^1(\Oo) \hookrightarrow H^s(\Oo)$,
 that is, there exists $\mathbf{c} = \mathbf{c}(n,s) > 0$ s.t.
 \begin{equation} \label{eq:H1embeddingHs}
  \iint_{\Oo\times\Oo}\frac{|u(x)-u(y)|^2}{|x-y|^{n+2s}}\,dx\,dy \leq 
  \mathbf{c}\,\|u\|_{H^1(\Oo)}^2\quad\text{for every $u\in H^1(\Oo)$}.
 \end{equation}
  In particular, if $\Oo\subseteq\R^n$ is a
  \emph{bounded open set} (with no regularity as\-sump\-tions on $\partial\Oo$) and if
  $u\in H_0^1(\Oo)$, setting $\hat{u} = u\cdot\mathbf{1}_\Oo\in H^1(\R^n)$ we have
  \begin{equation} \label{eq:H01embeddingHs}
  \iint_{\R^{2n}}\frac{|\hat{u}(x)-\hat{u}(y)|^2}{|x-y|^{n+2s}}\,dx\,dy \leq 
  \beta\,\int_\Oo|\nabla u|^2\,dx,
 \end{equation}
 where $\beta > 0$ is a suitable constant depending on $n,s$ and on $|\Omega|$. Here
 and throughout,  $|\cdot|$ denotes the $n$-dimensional Lebesgue measure.
\end{itemize}
The relation between the space $H^s(\Oo)$ and the fractional Laplacian $(-\Delta)^s$
is rooted in the following \emph{fractional integration-by-parts formula}: let 
$\Oo,\mathcal{V}\subseteq\R^n$ be open sets (bounded or not) such that $\Oo\Subset\mathcal{V}$,
and let $u\in \mathcal{L}^s(\R^n)\cap C^2(\mathcal{V}) \subseteq H^s(\Oo)$
(see the above \eqref{eq:H1embeddingHs}); given any $\varphi\in C_0^\infty(\Omega)$, we have
\begin{equation*}
\begin{split}
 \int_{\Oo}(-\Delta)^su\,\varphi\,dx & = \iint_{\R^{2n}}\frac{(u(x)-u(y))(\varphi(x)-\varphi(y))}{|x-y|^{n+2s}}\,dx\,dy \\
 & = 
 \iint_{\Oo\times\Oo}\frac{(u(x)-u(y))(\varphi(x)-\varphi(y))}{|x-y|^{n+2s}}\,dx\,dy
 \\
 &\qquad\quad-2\,\iint_{(\R^n\setminus\Oo)\times\Oo}
 \frac{(u(x)-u(y))\varphi(y)}{|x-y|^{n+2s}}\,dx\,dy
 \end{split}
\end{equation*}
(notice that, since $u\in \mathcal{L}^s(\R^n)\cap C^2(\mathcal{V})$, we 
have $(-\Delta)^su\in C(\mathcal{V})$).

As a con\-se\-quence of the above formula, it is then natural to define the \emph{fractional Laplacian} $(-\Delta)^s u$
of a function $u\in \mathcal{L}^s(\R^n)\cap H^s(\Oo)$ as the
\emph{linear functional} acting on the space $C_0^\infty(\Oo)$ as follows
\begin{equation} \label{eq:fractionalLapWeak}
 (-\Delta)^s u(\varphi) = \iint_{\R^{2n}}\frac{(u(x)-u(y))
 (\varphi(x)-\varphi(y))}{|x-y|^{n+2s}}\,dx\,dy.
 \end{equation}
 Since we are assuming that $u\in\mathcal{L}^s(\R^n)\cap H^s(\Oo)$, and since the kernel
 $|z|^{-n-2s}$ is in\-te\-gra\-ble at infinity, it is easy to recognize 
 that the functional $(-\Delta)^s u$ is actually a \emph{distribution on $\Oo$}; more precisely,
 for every compact set $K\subseteq\Oo$ there exists a constant
 $C > 0$ such that
 \begin{equation} \label{eq:contDeltasu}
  \begin{gathered}
  \iint_{\R^{2n}}\frac{|u(x)-u(y)|
 |\varphi(x)-\varphi(y)|}{|x-y|^{n+2s}}\,dx\,dy 
 \leq C\big(1+d(K,\R^n\setminus\Oo)^{-2s}\big)\|\varphi\|_{H^s(\Oo)}
  \\[0.1cm]
  \text{for every $\varphi\in C_0^\infty(\Oo)$ with $\mathrm{supp}(\varphi)\subseteq K$}
  \end{gathered} 
 \end{equation}
 (here, the constant $C$ depends on $n,s$ and on $u$), and this is enough to guarantee
 that $(-\Delta)^s u\in \mathcal{D}'(\Omega)$ (as $\|\varphi\|_{H^s(\Omega)}
 \leq c\,\|\varphi\|_{C^1(K)}$ for some absolute constant $c > 0$).
 In the particular case when $u\in H^s(\R^n)\subseteq \mathcal{L}^s(\R^n)$,
  the above \eqref{eq:contDeltasu} shows that the distribution
  $(-\Delta)^s u$ can be continuously extended to 
  $H^s(\R^N)$
  (recall that $C_0^\infty(\R^n)$
  is \emph{dense} in $H^s(\R^n)$), and thus
  $$(-\Delta)^s u\in (H^s(\R^n))'.$$ 
  In this case, we also have
  $$\frac{\mathrm{d}}{\mathrm{d}t}\Big|_{t = 0}[u+tv]^2_{s,\R^n} = (-\Delta)^s u(v)\quad
  \text{for all $v\in H^s(\R^n)$}.$$
  \vspace*{0.1cm}
  
  \noindent\textbf{ii) The space $\mathcal{X}^{1,2}(\Omega)$.}
  Now we have briefly recalled some basic facts
  from the Weak Theory for $(-\Delta)^s$, we are in a position
  to introduce the adequate functional setting for the study of mixed local-nonlocal operators.
  \vspace*{0.1cm}
  
  Let then $\Omega\subseteq\R^n$ be a bounded open set Lipschitz boundary $\de\Omega$.
  In view of the usual definition of the (local) Sobolev space $H_0^1(\Omega)$,
  and taking into account \eqref{eq:fractionalLapWeak}, we 
  define the space
  $\mathcal{X}^{1,2}(\Omega)$ as the com\-ple\-tion
  of $C_0^\infty(\Omega)$ w.r.t.\,the \emph{global norm} 
  $$\rho(u) := \left(\||\nabla u|\|^2_{L^2(\R^n)}+[u]^2_{s,\R^n}\right)^{1/2},
  \qquad u\in C_0^\infty(\Omega).$$
  Due to its relevance in the sequel, we also introduce a distinguished notation
  for the \emph{cone of the non-negative functions} in $\mathcal{X}^{1,2}(\Omega)$: we set
  $$\mathcal{X}^{1,2}_+(\Omega) := \{u\in\mathcal{X}^{1,2}(\Omega):\,\text{$u\geq 0$ a.e.\,in $\Omega$}\}.$$
 %
  %
  Since this norm $\rho$ is induced by the scalar product
    $$\mathcal{B}(u,v) := \int_{\R^n}\nabla u\cdot\nabla v\,dx
    + \langle u,v\rangle_{s,\R^n}$$
  (where $\cdot$ denotes the usual scalar product in 
    $\R^n$), the space $\mathcal{X}^{1,2}(\Omega)$
    is a real \emph{Hilbert space}; most importantly, since $\Omega$ is bounded
    and $\de\Omega$ is Lipschitz, by combining the above
    \eqref{eq:H1embeddingHs} with the classical Poincaré
    inequality we infer that
    \begin{equation*}
     \vartheta^{-1}\|u\|_{H^1(\R^n)}\leq \rho(u)\leq \vartheta\|u\|_{H^1(\R^n)}\qquad
    \text{for every $u\in C_0^\infty(\Omega)$},
    \end{equation*}
    where $\vartheta > 1$ is a suitable constant depending on $n,s$ and on $|\Omega|$.
    Thus, $\rho(\cdot)$ and the full $H^1$-norm in $\R^n$
   are \emph{actually equivalent} on the space $C^\infty_0(\Omega)$, so that
   \begin{equation} \label{eq:defX12explicit}
   \begin{split}
    \mathcal{X}^{1,2}(\Omega) & = \overline{C_0^\infty(\Omega)}^{\,\,\|\cdot\|_{H^1(\R^n)}} \\
    & = \{u\in H^1(\R^n):\,\text{$u|_\Omega\in H_0^1(\Omega)$ and 
    $u\equiv 0$ a.e.\,in $\R^n\setminus\Omega$}\}.
    \end{split}
   \end{equation}
   We explicitly observe that, on account of \eqref{eq:defX12explicit}, 
   the functions in $\mathcal{X}^{1,2}(\Omega)$ naturally
  satisfy the nonlocal Dirichlet condition 
  prescribed in problem \eqref{eq:Problem}$_\lambda$, that is,
  \begin{equation*}
   \text{$u\equiv 0$ a.e.\,in $\R^n\setminus\Omega$ for every $u\in\mathcal{X}^{1,2}(\Omega)$}.
   \end{equation*}
   \begin{remark}[Properties of the space $\mathcal{X}^{1,2}(\Omega)$] \label{rem:spaceX12}
    For a future reference, we list in this remark some properties
    of the function space $\mathcal{X}^{1,2}(\Omega)$ which will be repeatedly
    exploited in the rest of the paper.
    \begin{enumerate}
     \item[1)] Since both $H^1(\R^n)$ and $H_0^1(\Omega)$ are \emph{closed} under the
     maximum/minimum o\-pe\-ra\-tion, it is readily seen that
     $$u_{\pm}\in \mathcal{X}^{1,2}(\Omega)\quad\text{for every $u\in\mathcal{X}^{1,2}(\Omega)$},$$
     \noindent where $u_{+}= \max\{u,0\}$ and $u_{-}=\max\{-u,0\}$.
     \item[2)] On account of \eqref{eq:H01embeddingHs}, for every
     $u\in\mathcal{X}^{1,2}(\Omega)$ we have
     \begin{equation} \label{eq:X12HsRn}
    [u]_{s,\R^n}^2 = 
     \iint_{\R^{2n}}\frac{|u(x)-u(y)|^2}{|x-y|^{n+2s}}\,dx\,dy \leq\beta\int_\Omega|\nabla u|^2\,dx.
   \end{equation}
   As a consequence,
   the norm $\rho_{\e}$ is \emph{glo\-bal\-ly equivalent} on $\mathcal{X}^{1,2}(\Omega)$ to the 
   $H_0^1$-no\-rm: in fact, by \eqref{eq:X12HsRn} there exists a constant $\Theta
   = \Theta_{n,s} > 0$ such that
   \begin{equation} \label{eq:equivalencerhoH01}
    \||\nabla u|\|_{L^2(\Omega)}\leq \rho(u)\leq \Theta\||\nabla u|\|_{L^2(\Omega)}
    \quad\text{for every $u\in\mathcal{X}^{1,2}(\Omega)$}.
   \end{equation}
   
   \item[3)] By the (local) Sobolev inequality, for every $u\in\mathcal{X}^{1,2}(\Omega)$ 
   we have
   \begin{align*}
    S_n\|u\|_{L^{2^*}(\Omega)}^2 & = \|u\|_{L^{2^*}(\R^n)}^2
    \leq \int_{\R^n}|\nabla u|^2\,dx\leq \rho(u)^2.
   \end{align*}
   This, together with H\"older's inequality
   (recall that $\Omega$ is \emph{bounded}), proves the \emph{continuous embedding}
   $\text{$\mathcal{X}^{1,2}(\Omega)\hookrightarrow L^{p}(\Omega)$ for every $1\leq p\leq 2^*$}.$
   \vspace*{0.1cm}
   
   \item[4)] By combining \eqref{eq:equivalencerhoH01} with the \emph{compact embedding} of 
   $H_0^1(\Omega)\hookrightarrow L^p(\Omega)$ (holding true for every $1\leq p < 2^*$),
   we derive that also the embedding
   $$\text{$\mathcal{X}^{1,2}(\Omega)\hookrightarrow L^{p}(\Omega)$ is compact
   for every $1\leq p< 2^*$}.$$
   As a consequence, if $\{u_k\}_k$ is a bounded sequence in $\mathcal{X}^{1,2}(\Omega)$, it is possible
   to find a (unique) function $u\in\mathcal{X}^{1,2}(\Omega)$ such that (up to a sub-sequence)
   \begin{itemize}
    \item[a)] $u_n\to u$ weakly in $\mathcal{X}^{1,2}(\Omega)$;
    \item[b)] $u_n\to u$ \emph{strongly} in $L^p(\Omega)$ for every $1\leq p < 2^*$;
    \item[c)] $u_n\to u$ pointwise a.e.\,in $\Omega$.
   \end{itemize}
   Clearly, since both $u_n$ (for all $n\geq 1$) and $u$ \emph{identically vanish}
    out of $\Omega$, see
   \eqref{eq:defX12explicit}, we can replace
   $\Omega$ with $\R^n$ in the above assertions b)-c).
    \end{enumerate}
       \end{remark}
     We now observe that, since we aim
    to consider (a weak realization of) the $\e$-de\-pen\-dent operator
    $\LL = -\Delta+\e(-\Delta)^s$
    (the latter being the leading operator
    of problem \eqref{eq:Problem}$_\lambda$), it follows once again
    from \eqref{eq:fractionalLapWeak}
    that the bilinear form naturally associated with $\LL$ is the following
    $$\mathcal{B}_{\e}(u,v) = \int_{\R^n}\nabla u\cdot\nabla v\,dx
    + \e\,\langle u,v\rangle_{s,\R^n};$$
    in its turn, this form $\mathcal{B}_\e$ induces the $\e$-dependent quadratic form
    $$\rho_\e(u) = \||\nabla u|\|^2_{L^2(\R^n)}+\e\,[u]^2_s\qquad
     (\text{for $u\in\mathcal{X}^{1,2}(\Omega)$}).$$
    While in this perspective it should seem more \emph{natural} to use
     the
    norm $\rho_\e$ in place of $\rho$ on the space $\mathcal{X}^{1,2}(\Omega)$, 
    it is readily seen that these two norms
    are indeed \emph{equivalent on $\mathcal{X}^{1,2}(\Omega)$}
    (and equivalent to the $H_0^1$-norm), \emph{uniformly with respect to $\e$}:
    in fact, taking into account \eqref{eq:equivalencerhoH01} (and since $0<\e\leq 1$), we have
    \begin{equation} \label{eq:equivalenceuniforme}
     \||\nabla u|\|_{H^1_0(\Omega)}\leq\rho_\e(u)\leq \rho(u)\leq \Theta\||\nabla u|\|_{H^1_0(\Omega)}
     \end{equation}
   for some $\Theta > 0$ only depending on $n,s$. On account of \eqref{eq:equivalenceuniforme},
   we can indifferently use $\rho_\e(\cdot),\,\rho(\cdot)$ and the $H_0^1$-norm
   to define the topology of the space $\mathcal{X}^{1,2}(\Omega)$, and this choice
   does not produce any dependence on $\e\in(0,1]$.
   \medskip
   
    \noindent\textbf{Notation.} We conclude this second part of the section
    with a short list of notation, which will be used in the sequel;
    here, as usual, $\e\in(0,1]$ is a fixed parameter.
   \vspace*{0.1cm}
   
    1)\,\,Given any open set $\mathcal{O}\subseteq\R^n$ (not necessarily bounded), we set
    \begin{equation*}
    \begin{split}
     \mathrm{a)}&\,\,\mathcal{B}_{\e,\mathcal{O}}(u,v) =
    \int_{\mathcal{O}}\nabla u\cdot\nabla v\,dx
    + \e\, \langle u,v\rangle_{s,\R^n}\,\,(\text{for $u,v\in\mathcal{X}^{1,2}(\Omega)$}); \\
    \mathrm{b)}&\,\,\mathcal{Q}_{\e,\Oo}(u) = \mathcal{B}_{\rho_{\e},\Oo}(u,u)\,\,(\text{for $u
    \in\mathcal{X}^{1,2}(\Omega)$}).
    \end{split}
    \end{equation*}
    Since $\mathcal{X}^{1,2}(\Omega)\subseteq H^1(\R^n)$
    (see \eqref{eq:defX12explicit}), 
    the above forms $\mathcal{B}_{\e,\mathcal{O}}$ and $\mathcal{Q}_{\e,\Oo}$ are 
    well-de\-fined; moreover,
    again by taking into account \eqref{eq:defX12explicit} we have
	\begin{itemize}
	 \item $\mathcal{B}_{\e,\Omega}(u,v) = \mathcal{B}_{\e,\R^n}(u,v)\equiv \mathcal{B}_\e(u,v)$
	 for all $u,v\in\mathcal{X}^{1,2}(\Omega)$;
	 \item $\mathcal{Q}_{\e,\Omega}(u) = \mathcal{Q}_{\e,\R^n}(u) 
	 \equiv \rho_{\e}(u)$ for all $u\in\mathcal{X}^{1,2}(\Omega)$.
	\end{itemize}

    2)\,\,Given any \emph{bounded} open set $\mathcal{O}\subseteq\R^n$, we set
   $$\|u\|_{H_0^1(\Oo)} := \||\nabla u|\|_{L^2(\Oo)} = \Big(\int_\Oo|\nabla u|^2\,dx\Big)^{1/2}.$$
     
  \noindent\textbf{iii) Weak sub/supersolutions of problem \eqref{eq:Problem}$_\lambda$.} Thanks to
  all the preliminaries recalled so far, we are finally ready to give
  the precise definition of \emph{weak sub/supersolutions} of problem  \eqref{eq:Problem}$_\lambda$.
  Actually, for a reason which will be clear in Section \ref{sec:weakharnack}, we consider
  the more general problem
  \begin{equation} \label{eq:pbgeneral}
    \begin{cases}
     \LL u = \frac{\lambda}{u^\gamma}+f(x,u) & \text{in $\Omega$} \\
     u > 0 & \text{in $\Omega$}, \\
     u = 0 & \text{in $\R^n\setminus\Omega$}
     \end{cases}
    \end{equation}
  where $f:\Omega\times(0,+\infty)\to\R$ is an arbitrary Carath\'eodory function satisfying
  the following {growth condition}: {\em there exists a constant $K_f > 0$ such that}
  \begin{equation} \label{eq:generalgrowthf}
   |f(x,t)|\leq K_f(1+t^{2^*-1})\quad\text{\emph{for a.e.\,$x\in\Omega$ and every $t > 0$}}.
   \end{equation}
   Clearly, problem \eqref{eq:Problem}$_\lambda$ is of the form \eqref{eq:pbgeneral}, with the choice
   $f(x,t) = t^{2^*-1}$.
  \begin{definition} \label{def:weaksol}
   Let $f:\Omega\times(0,+\infty)\to\R$
   be a Carath\'eodory function sa\-ti\-sfying the above condition
   \eqref{eq:generalgrowthf}.
   We say that a function $u\in\mathcal{X}^{1,2}(\Omega)$ is a \emph{weak subsolution}
   (resp.\,\emph{supersolution}) of problem \eqref{eq:pbgeneral}
   if it satisfies the following properties:
   \begin{itemize}
    \item[i)] $u > 0$ a.e.\,in $\Omega$ and $u^{-\gamma}\in L^1_{\mathrm{loc}}(\Omega)$;
    \vspace*{0.1cm}
    
    \item[ii)] for every test function $\varphi\in C_0^\infty(\Omega),\,\text{$\varphi\geq 0$ in $\Omega$}$, 
    we have
    \begin{equation} \label{eq:weakformSol}
     \mathcal{B}_{\e}(u,\varphi) \leq\,[\text{resp.}\,\geq]\,\,
    \lambda\int_{\Omega}u^{-\gamma}\varphi\,dx
    + \int_{\Omega}f(x,u)\varphi\,dx.
    \end{equation}
   \end{itemize}
   We say that $u$ is a \emph{weak solution} of problem
   \eqref{eq:pbgeneral} if it is both a weak subsolution and a weak
   supersolution of the same problem.
  \end{definition}
  \begin{remark} \label{rem:defweaksolPb}
  We now list here below some comments concerning the above Definition \ref{def:weaksol}. In what follows,
  we tacitly understand that $f:\Omega\to(0,+\infty)\to\R$ is a Carath\'eodory function satisfying
  the growth
  assumption \eqref{eq:generalgrowthf}.
  \medskip
  
  1)\,\,If $u\in\mathcal{X}^{1,2}(\Omega)$ is a weak sub/supersolution of problem
  \eqref{eq:pbgeneral}, all the integrals appearing in \eqref{eq:weakformSol} 
  \emph{exist and are finite}.
  Indeed, by combining
  \eqref{eq:X12HsRn} with the H\"older inequality, for every 
  $\varphi\in C_0^\infty(\Omega),\,\text{$\varphi\geq 0$ in $\Omega$}$, 
  we obtain (recall that $\e \in (0,1]$)
  \begin{equation} \label{eq:BrhouphiFinite}
  \begin{split}
   |\mathcal{B}_{\e}(u,\varphi)| & \leq 
   \| u\|_{H_0^1(\Omega)}\|\varphi\|_{H_0^1(\Omega)}+
   [u]_{s,\R^n}\cdot[\varphi]_{s,\R^n} \\
   & \leq (1+\beta^2)\| u\|_{H_0^1(\Omega)}\|\varphi\|_{H_0^1(\Omega)} < +\infty.
   \end{split}
  \end{equation}
  Moreover, using \eqref{eq:generalgrowthf} and the Sobolev inequality, we also get
  \begin{equation} \label{eq:termfFinite}
   \begin{split}
    \int_\Omega|f(x,u)|\cdot|\varphi|& \leq K_f\Big(\|\varphi\|_{L^1(\Omega)}
    + \int_{\Omega}|u|^{2^*-1}|\varphi|\,dx\Big) \\
    & (\text{using H\"older's inequality}) \\
    & \leq  K_f\big(\|\varphi\|_{L^1(\Omega)}
    +\|u\|_{L^{2^*}(\Omega)}^{2^*-1}\cdot\|\varphi\|_{L^{2^*}(\Omega)}\big) \\
    & \leq K_f\big(\|\varphi\|_{L^1(\Omega)}
    +S_n^{-2^*/2}\| u\|_{H_0^1(\Omega)}^{2^*-1}\cdot\| \varphi\|_{H_0^1(\Omega)}\big) \\
    & (\text{again by H\"older's inequality and Poincaré inequality}) \\
    & \leq C\| \varphi\|_{H_0^1(\Omega)}\big(1+\| u\|_{H_0^1(\Omega)}^{2^*-1}\big)<+\infty,
   \end{split}
  \end{equation}
  where $S_n > 0$ is the best Sobolev constant in $\R^n$, and $C > 0$ depends on $n,\,f,\,|\Omega|$.
  Fi\-nally, since \emph{we are assuming that $u^{-\gamma}\in L^1_{\mathrm{loc}}(\Omega)$}, we obviously
  have
  $$0\leq \int_\Omega u^{-\gamma}\varphi\,dx \leq 
  \|\varphi\|_{L^\infty(\Omega)}\int_{\mathrm{supp}(\varphi)}u^{-\gamma}\,dx < +\infty. 
  $$
  We explicitly stress that this last estimate 
  (which is related with the \emph{singular term} $u^{-\gamma}$) is the unique estimate involving
  the $L^\infty$-norm of the test function $\varphi$; on the contrary, 
  estimates \eqref{eq:BrhouphiFinite}-\eqref{eq:termfFinite} only involve
  the \emph{$H_0^1$-norm} of $\varphi$.
  \vspace*{0.1cm}
  
  2)\,\,If $u\in\mathcal{X}^{1,2}(\Omega)$ is a weak \emph{solution} of
  problem
  \eqref{eq:pbgeneral}, and if $\varphi\in C_0^\infty(\Omega)$ is a {non-negative}
  test function (that is, $\varphi\geq 0$ in $\Omega$), by \eqref{eq:BrhouphiFinite}-\eqref{eq:termfFinite}
  we have
  \begin{align*}
   0\leq \int_\Omega u^{-\gamma}\varphi\,dx & = 
   \frac{1}{\lambda}\Big(\mathcal{B}_{\e}(u,\varphi)
   + \int_\Omega f(x,u)\varphi\,dx\Big)
   \\
   & \leq C\| \varphi\|_{H_0^1(\Omega)}\big(1+\| u\|_{H_0^1(\Omega)}^{2^*-1}\big);
  \end{align*}
  from this, by using a standard \emph{density argument},
  and by taking into account Remark \ref{rem:spaceX12}\,-1), we can easily prove the following facts:
  \begin{itemize}
   \item[a)] $u^{-\gamma}\varphi\in L^1(\Omega)$
  \emph{for every $\varphi\in\mathcal{X}^{1,2}(\Omega)$};
  \item[b)] identity \eqref{eq:weakformSol}
  actually holds \emph{for every $\varphi\in\mathcal{X}^{1,2}(\Omega)$}, see \cite{GarainUkhlov}.
  \end{itemize}

   3)\,\,Let $u\in\mathcal{X}^{1,2}(\Omega)$ be a weak \emph{solution} of problem \eqref{eq:pbgeneral}.
  Since, in particular, we know that
  $u > 0$ a.e.\,in $\Omega$ (and $u\equiv 0$ a.e.\,in $\R^n\setminus\Omega$), it is quite easy
  to recognize that $u$ is a \emph{weak supersolution of the equation}
  $$\LL u  = 0\quad\text{in $\Omega$},$$
  in the sense of \cite[Definition 2.5]{GarainKinnunen} (see also
  \cite[Remark 2.6]{GarainKinnunen}
  and note that, by 
  \eqref{eq:H1embeddingHs} and the definition
  of $\mathcal{X}^{1,2}(\Omega)$, we have $u\in H^1(\R^n)\subseteq H^s(\R^n)$).
  As a consequence, we are entitled to apply \cite[Lemma 8.1]{GarainKinnunen}, ensuring that
  \begin{equation*}
    \begin{gathered}
     \text{for every $\mathcal{O}\Subset\Omega$ there exists $c(\mathcal{O},u) > 0$ 
     s.t.\,$u\geq c(\mathcal{O},u) > 0$ a.e.\,in $\mathcal{O}$}.
    \end{gathered}
  \end{equation*}
  We explicitly notice that, since the function $u$ intrinsically depends on $\e$
  (as it is a weak supersolution of an equation depending on
  the parameter $\e$), the constant $c$
  in the above estimate is not, in general, independent of $\e$.
  \end{remark}
  \section{Some auxiliary results} \label{sec:weakharnack}
  Before embarking on the proof of Theorem \ref{thm:main}, we e\-sta\-blish in this section
  a couple of auxiliary results which will be fundamental in our arguments.
  \medskip
  
  \noindent\textbf{i)\,\,Weak Harnack-type inequality for $\LL+a(x)$}.
  The first auxiliary result we aim to prove is a
  we\-ak Harnack-type inequality for the weak supersolutions of 
  \begin{equation} \label{eq:withZeroOrder}
   \text{$\LL u + a(x)u = 0$ in $\Omega$},
   \end{equation}
  where $a\in L^\infty(\Omega),\,\text{$a\geq 0$ a.e.\,in $\Omega$}$. As already
  discussed in the introduction, such a result turns out to be a proper substitute
  of \cite[Theorem 3]{BNH1C1}, and it will be used as a \emph{key tool} in the proof of Theorem \ref{thm:main}-b).
  \medskip
  
  To begin with, since equation \eqref{eq:withZeroOrder} \emph{is not} a particular
  case of the one appearing in problem
  \eqref{eq:pbgeneral}, we give the following definition.
  \begin{definition} \label{def:weaksubsolZeroOrder}
   Let $a\in L^\infty(\Omega)$. We say that a function
   $u\in \mathcal{L}^s(\R^n)\cap H^1_{\mathrm{loc}}(\Omega) $ is a
   \emph{weak subsolution} (resp.\,\emph{supersolution}) of equation \eqref{eq:withZeroOrder} if
   \begin{equation} \label{eq:defweaksubsuperzeroOrd}
    \mathcal{B}_{\e,\mathcal{O}}(u,\varphi)+\int_{\mathcal{O}}
    a(x)u\varphi\,dx \leq\,[\text{resp}.\,\geq]\,\,0
   \end{equation}
   for every $\mathcal{O}\Subset\Omega$ and every $\varphi\in\mathcal{X}^{1,2}_+(\mathcal{O})$.
  \end{definition}
  \begin{remark} \label{rem:defbenposta}
   By combining \eqref{eq:H1embeddingHs}-\eqref{eq:H01embeddingHs} with
   \eqref{eq:contDeltasu}, it is easy to recognize that
   the above Definition \ref{def:weaksubsolZeroOrder} \emph{is well-posed}. In fact,
   let $\mathcal{V}$ be a bounded open set \emph{with smooth boundary} and such that
   $\mathcal{O}\Subset\mathcal{V}\Subset\Omega$.
   Since $u\in H^1(\mathcal{\mathcal{V}})$ {(}recall that $u\in H^1_{\mathrm{loc}}(\Omega)$ and 
 that $\mathcal{V}\Subset\Omega${)}, from \eqref{eq:H1embeddingHs} we deduce that
 $u\in H^s(\mathcal{V})$; as a consequence, we are then entitled to
 apply \eqref{eq:contDeltasu} with the choice $K = \overline{\Oo}\Subset\mathcal{V}$: this gives   
   \begin{equation*}
  \begin{gathered}
  \iint_{\R^{2n}}\frac{|u(x)-u(y)|
 |\varphi(x)-\varphi(y)|}{|x-y|^{n+2s}}\,dx\,dy 
 \leq C\big(1+d(\overline{\Oo},\R^n\setminus\mathcal{V})^{-2s}\big)\|\varphi\|_{H^s(\mathcal{V})}
  \end{gathered} 
 \end{equation*}
 and this estimate holds {for every $\varphi\in C_0^\infty(\Oo)$} (here, $C > 0$ is a constant
 depending on $n,s$ and on $u$).
 Now, by combining this last estimate with 
 \eqref{eq:H01embeddingHs},
 we get
 \begin{equation*}
  \iint_{\R^{2n}}\frac{|u(x)-u(y)||\varphi(x)-\varphi(y)|}{|x-y|^{N+2s}}\,dx\,dy
  \leq C'\big(1+d(\overline{\Oo},\R^n\setminus\mathcal{V})^{-2s}\big)
  \|\varphi\|_{H^1_0(\Oo)},
 \end{equation*}
 for a suitable constant $C' > 0$;
 from this, by taking into account that $C_0^\infty(\Oo)$ is dense in $\mathcal{X}^{1,2}(\Oo)
 \subseteq H^1(\R^n)$, we
 easily
 derive that
 $$|\mathcal{B}_{\e,\mathcal{O}}(u,\varphi)|<+\infty\quad
 \text{for every $\varphi\in \mathcal{X}^{1,2}(\mathcal{O})$}.$$
 Finally, since $a\in L^\infty(\Omega)$, we also have
 $$\int_\mathcal{O}|a(x)u\varphi|\,dx\leq 
 \|a\|_{L^\infty(\Omega)}\|u\|_{L^2(\mathcal{O})}\|\varphi\|_{L^2(\mathcal{O})}<+\infty$$
 and this proves that Definition \ref{def:weaksubsolZeroOrder} is well-posed.
  \end{remark}
  With Definition \ref{def:weaksubsolZeroOrder} at hand, we are ready to state the
  announced weak Harnack-ty\-pe inequality for weak supersolutions of equation \eqref{eq:withZeroOrder}.
  Throughout what follows, if $u\in\mathcal{L}^s(\R^n)$
 and if $B_r(x_0)\subseteq\R^n$ is a given ball, we define 
\begin{equation*}
\mathrm{Tail}(u; x_0, r):= r^2 \int_{\mathbb{R}^n \setminus B_r(x_0)}\dfrac{|u(y)|}{|y-x_0|^{n+2s}}\, dy<+\infty.
\end{equation*}  
This quantity is usually referred to as the \emph{tail of $u$} (with respect to $B_r(x_0)$).
\begin{proposition} \label{prop:weakHarnack}
Let $a\in L^\infty(\Omega),\,\text{$a\geq 0$ a.e.\,in $\Omega$}$. Moreover,
let $u$ be a weak su\-per\-solution of 
equation \eqref{eq:withZeroOrder} such that 
$u \geq 0$ a.e.\,in $B_{R}(x_0)\Subset\Omega$.

Then, there exist $Q = Q(n,s,a)>0$ and $c=c(n,s,a)>0$, \emph{both independent of
the fixed $\e\in(0,1]$}, such that
\begin{equation}\label{eq:8.1GK}
\left(\fint_{B_r(x_0)}u^Q \, dx\right)^{1/Q} \leq c \essinf_{B_r(x_0)}u + c \left(\dfrac{r}{R}\right)^{2}\mathrm{Tail}(u_{-};x_0,R),
\end{equation}
\noindent whenever $B_r(x_0) \subset B_R(x_0)$ and $r\in (0,1]$.
\end{proposition}
In order to prove Proposition \ref{prop:weakHarnack} we 
first establish the following \emph{Caccioppoli-type inequality} for $\LL +a(x)$, 
which is modeled on \cite[Lemma 3.1]{GarainKinnunen}.
\begin{lemma}  \label{lem:Caccioppolitype}
Let $a\in L^\infty(\Omega)$, and
let $u$ be a weak subsolution \emph{[}resp.\,supersolution\emph{]} of 
equation \eqref{eq:withZeroOrder}.
We arbitrarily fix $k\in\R$, and we set
$$w= (u-k)_+\quad[\text{resp.}\,w = (u-k)_-].$$ 
Then, there exists $C>0$, \emph{independent of $\e\in(0,1]$}, such that 
\begin{equation}\label{eq:CaccioppoliType}
\begin{aligned}
&\int_{B_{r}(x_0)}  \psi^2 |\nabla w|^2 \, dx +\e\, \iint_{B_{r}(x_0)\times B_{r}(x_0)} \dfrac{|w(x)\psi^2(x) - w(y)\psi^2(y)|^2}{|x-y|^{n+2s}}\, dx dy \\
& \qquad\qquad +\,\,[\text{resp.}\,-]\int_{B_r(x_0)}a(x)uw\psi^2\,dx\\
& \qquad\leq C 
 \bigg( \int_{B_{r}(x_0)}w^2|\nabla \psi|^2 \, dx \\
 & \qquad\qquad 
  + 
   \iint_{B_{r}(x_0)\times B_{r}(x_0)}\max\{w(x),w(y)\} \dfrac{|\psi(x)- \psi(y)|^2}{|x-y|^{n+2s}}\, dx dy \\
&\qquad\qquad+ \esssup_{x \in \mathrm{supp}(\psi)} 
\int_{\mathbb{R}^n \setminus B_{r}(x_0)}\dfrac{w(y)}{|x-y|^{n+2s}}\, dy \, \int_{B_{r}(x_0)}w\psi^2 \, dx\bigg),
\end{aligned}
\end{equation}
\noindent whenever $B_r(x_0) \Subset \Omega$ and $\psi \in C^{\infty}_{0}(B_r(x_0))$ is nonnegative.
\end{lemma}
\begin{proof}
 Assume that $u$ is a weak \emph{subsolution} of equation \eqref{eq:withZeroOrder}, and let $w,\psi$
 be as in the statement. Since $u\in H^1_{\mathrm{loc}}(\Omega)$ and $\psi\in C_0^\infty(B_r(x_0))$,
 we clearly have
 $$\varphi := w\psi^2\in\mathcal{X}^{1,2}_+(B_r(x_0)).$$
 As a consequence of this fact, we are then entitled to use this function $\varphi$ as a test
 function in \eqref{eq:defweaksubsuperzeroOrd} (recall that $\Oo = B_r(x_0)\Subset\Omega$), 
 thus obtaining
 \begin{equation} \label{eq:usewpsitest}
  \begin{split}
  & 0\geq B_{\e,B_r(x_0)}(u,\varphi)+\int_{B_r(x_0)}a(x)u\varphi\,dx \\
  & \qquad
  = \int_{B_r(x_0)}\nabla u\cdot\nabla (w\psi^2)\,dx
  \\
  & \qquad\qquad +\e\, \iint_{\R^{2n}}\frac{(u(x)-u(y))((w\psi^2)(x)-(w\psi^2)(y))}{|x-y|^{n+2s}}\,dx\,dy
  \\
  & \qquad\qquad\qquad+\int_{B_r(x_0)}a(x)uw\psi^2\,dx \\
  & \qquad \equiv I+J+\int_{B_r(x_0)}a(x)uw\psi^2\,dx.
 \end{split}
 \end{equation}
 Now, they are proved in \cite[Lemma 3.1]{GarainKinnunen} the following
 estimates for the integrals $I$ and $J$, holding true \emph{for every function $u\in \LL^s(\R^n)\cap 
 H^1_{\mathrm{loc}}(\Omega)$}:
 \begin{align*}
   \mathrm{i)}&\,\,I\geq c_1\int_{B_r(x_0)}\psi^2|\nabla w|^2\,dx - C_1\int_{B_r(x_0)}w^2|\nabla \psi|^2\,dx \\
  \mathrm{ii)}&\,\,J \geq \e \, c_2\iint_{B_{r}(x_0)\times B_{r}(x_0)}
  \dfrac{|w(x)\psi^2(x) - w(y)\psi^2(y)|^2}{|x-y|^{n+2s}}\, dx dy \\
  &\qquad
  -\e \, C_2\iint_{B_{r}(x_0)\times B_{r}(x_0)}\max\{w(x),w(y)\} \dfrac{|\psi(x)- \psi(y)|^2}{|x-y|^{n+2s}}\, dx dy \\
&\qquad\qquad-\e \, C_2\esssup_{x \in \mathrm{supp}(\psi)} 
\int_{\mathbb{R}^n \setminus B_{r}(x_0)}\dfrac{w(y)}{|x-y|^{n+2s}}\, dy\,
\int_{B_{r}(x_0)}w\psi^2 \, dx,
 \end{align*}
 for suitable \emph{absolute}, and independent of $\e$, constants $c_i,\,C_i > 0$ (for $i = 1,2$).
 By combining \eqref{eq:usewpsitest} with the above i)-ii), and recalling that $\e \leq 1$, we then obtain the claimed \eqref{eq:CaccioppoliType}.
 
 Finally, if $u$ is a weak \emph{supersolution}
 of equation \eqref{eq:withZeroOrder} it suffices to apply the ob\-ta\-i\-ned
 estimate to $v = -u$, which is a weak \emph{subsolution} of the same equation.
\end{proof}
 We now establish 
 an {\it expansion of positivity}-type result, which slightly extends \cite[Lemma 7.1]{GarainKinnunen} and whose proof heavily relies on that one. In order to avoid tiring 
 re\-pe\-ti\-tions, we will only highlight the few modifications needed.

\begin{lemma} \label{lem:expansion}
Let $a\in L^\infty(\Omega),\,\text{$a\geq 0$ a.e.\,in $\Omega$}$, and
let $u$ be a weak \emph{supersolution} of equation \eqref{eq:withZeroOrder} 
satisfying $u \geq 0$ 
in $B_{R}(x_0) \Subset \Omega$. Let then $k \geq 0$, and suppose there exists 
a constant $\tau \in (0,1]$ such that
\begin{equation*} 
|B_{r}(x_0) \cap \{ u \geq k\}| \geq \tau \, |B_r(x_0)|,
\end{equation*}
\noindent for some $r \in (0,1]$ satisfying $0 <r< R/16$. 
Then, there exists $\delta = \delta(n,s,\tau,a)>0$, 
\emph{independent of the fixed $\e\in(0,1]$}, such that
\begin{equation} \label{eq:explansionConclusion}
\essinf_{B_{4r}(x_0)}u \geq \delta k - \left(\dfrac{r}{R}\right)^{2} \mathrm{Tail}(u_{-};x_0,R).
\end{equation}
\end{lemma}
\begin{proof}
As in the proof of \cite[Lemma 7.1]{GarainKinnunen}, we proceed by steps.
\medskip

\textsc{Step 1:} In this step we prove that there exists $c=c(n,s,a)>0$ such that 
\begin{equation}\label{eq:7.3GK}
\begin{split}
& \left|B_{6r}(x_0) \cap \left\{ u \leq 2\delta k - \dfrac{1}{2}\left(\dfrac{r}{R}\right)^2 \mathrm{Tail}(u_{-};x_0,R) -\eta\right\}\right| \\
 &\qquad\qquad \leq\dfrac{c}{\tau \, \ln((2\delta)^{-1})} \, |B_{6r}(x_0)|,
\end{split}
\end{equation}
\noindent for every $\delta \in (0,\tfrac{1}{4})$ and for every $\eta>0$.
 We explicitly note that this is the analogous of \cite[Equation (7.3)]{GarainKinnunen}.
 To prove \eqref{eq:7.3GK}, we fix $\psi \in C^{\infty}_{0}(B_{7r}(x_0))$ such that 
\begin{itemize}
\item $0\leq \psi \leq 1$ in $B_{7r}(x_0)$;
\vspace*{0.05cm}
\item $\psi =1$ in $B_{6r}(x_0)$;
\vspace*{0.05cm}
\item $|\nabla \psi| \leq \tfrac{8}{r}$ in $B_{7r}(x_0)$. 
\end{itemize}
We then take $w:= u + t_{\eta}$, where
\begin{equation*}
t_{\eta}:= \dfrac{1}{2}\left(\dfrac{r}{R}\right)^2 \mathrm{Tail}(u_{-};x_0,R) + \eta,
\end{equation*}
\noindent and we notice that, since \emph{we assuming that $a\geq 0$ a.e.\,in $\Omega$},
this function $w$ is a weak \emph{supersolution} of 
equation \eqref{eq:withZeroOrder} as well. Therefore, 
by using the function 
$$\phi := w^{-1}\psi^2\in\mathcal{X}^{1,2}_+(B_{7r}(x_0))$$ 
as a test function
in \eqref{eq:defweaksubsuperzeroOrd} (with $u = w$), we get (recall $\e \leq 1$)
\begin{equation*}
\begin{aligned}
& 0 \leq \mathcal{B}_{\e,B_{7r}(x_0)}(w,\phi)+\int_{B_{7r}(x_0)}a(x)w\phi\,dx \\
& \qquad \leq \int_{B_{7r}(x_0)}\nabla w \cdot \nabla (w^{-1}\psi^2)\, dx \\
& \qquad\qquad + \iint_{B_{8r}(x_0)\times B_{8r}(x_0)}\dfrac{(w(x)-w(y))(\phi(x)-\phi(y))}{|x-y|^{n+2s}}\, dxdy\\
&\qquad\qquad +2 \int_{B_{8r}(x_0)}\int_{\mathbb{R}^{n}\setminus B_{8r}(x_0)}
\dfrac{(w(x)-w(y))\phi(x)}{|x-y|^{n+2s}}\,dx\,dy \\
& \qquad\qquad+ \int_{B_{7r}(x_0)}a(x)w(x)\dfrac{\psi^2(x)}{w(x)}\, dx\\
&\qquad \equiv I_1 + I_2 + I_3 + I_4.
\end{aligned}
\end{equation*}
Taking from \cite{GarainKinnunen} the estimates for the first three integrals, we are left with
\begin{equation*}
I_4 \leq \|a\|_{\infty}\int_{B_{8r}(x_0)}\psi^2(x)\, dx \leq c(n,\|a\|_{\infty}) r^{n-2},
\end{equation*}
\noindent where we have also used the fact that $r \leq 1$.
Combining this estimate with those of $I_i$ ($i=1,2,3$), we then get \cite[equation (7.11)]{GarainKinnunen} which in turn gives \eqref{eq:7.3GK}.
\medskip

\textsc{Step 2:} In this step we prove the following fact: 
\emph{for every $\eta>0$ there exists 
a positive constant $\delta = \delta(n,s,\tau) \in (0,1/4)$ such that }
\begin{equation}\label{eq:7.16GK}
\essinf_{B_{4r}(x_0)} u \geq \delta k - \left( \dfrac{r}{R}\right)^{2} \mathrm{Tail}(u_{-};x_0,R) -2\eta,
\end{equation}
 which is precisely \cite[equation (7.16)]{GarainKinnunen}.
To begin with, following \cite{GarainKinnunen} we can assume without loss of generality that
\begin{equation*}
\delta k \geq \left(\dfrac{r}{R}\right)^2 \mathrm{Tail}(u_{-};x_0,R)+2\eta,
\end{equation*}
\noindent otherwise \eqref{eq:7.16GK} is trivially satisfied (since $u\geq 0$ in $B_R(x_0)$). 
Let then $\varrho \in [r,6r]$ and let $\psi \in C_0^{\infty}(B_{\varrho}(x_0))$ be a cut-off function such that 
$$\text{$0\leq \psi\leq 1$ in $B_{\varrho}(x_0)$}.$$
We now arbitrarily fix a number $l\in (\delta k, 2\delta k)$, and we exploit the Caccioppoli-type
i\-ne\-qu\-ality in Lemma \ref{lem:Caccioppolitype} with $w = (u-l)_-\leq l+u_-$: this gives 
\begin{equation}\label{eq:7.18GK}
\begin{aligned}
&\int_{B_{\varrho}(x_0)}\psi^2 |\nabla w|^2 \, dx + \e\,
 \iint_{B_{\varrho}(x_0) \times B_{\varrho}(x_0)}\dfrac{|w(x)\psi(x)-w(y)\psi(y)|^2}{|x-y|^{n+2s}}\, dxdy\\
&\qquad \leq c \int_{B_{\varrho}(x_0)}w^2 |\nabla \psi|^2\, dx \\
& \qquad\qquad + c \iint_{B_{\varrho}(x_0) \times B_{\varrho}(x_0)}\max\{w(x),w(y)\}\dfrac{|\psi(x)-\psi(y)|^2}{|x-y|^{n+2s}}\, dxdy\\
&\qquad \qquad +c\,l\,|B_{\varrho}(x_0)\cap \{u<l\}|\cdot
 \esssup_{x\in \mathrm{supp}(\psi)}
 \int_{\mathbb{R}^n \setminus B_{\varrho}(x_0)}\dfrac{l+u_-(y)}{|x-y|^{n+2s}}\, dy \, \\
&\qquad\qquad + c \, \int_{B_{\varrho}(x_0)}a(x) uw\psi^2\, dx\\
&\qquad \equiv J_1 + J_2 + J_3 + c \, \int_{B_{\varrho}(x_0)}a(x) uw\psi^2\, dx.
\end{aligned}
\end{equation}
which is the analog of \cite[Equation (7.18)]{GarainKinnunen} in the present context.
With \eqref{eq:7.18GK} at hand, we apply \cite[Lemma 4.1]{DiBenedetto} to complete the proof
of \eqref{eq:7.16GK}. 
\vspace*{0.05cm}

To this aim, we have to introduce a bit of notation as it is used in \cite{GarainKinnunen}: for $j=0,1,\ldots$, we denote 
\begin{equation*}
l=k_j=\delta k + 2^{-j-1}\delta k, \quad \varrho=\varrho_j=4r + 2^{1-j}r, \quad \hat{\varrho}_j = 
\dfrac{\varrho_j + \varrho_{j+1}}{2}.
\end{equation*}
With these choices, and since $l\in (\delta k,2\delta k)$, for all $j\geq 0$ we have
\vspace*{0.1cm}

a)\,\,$\varrho_j, \hat{\varrho}_j \in [4r,6r]$ and $\varrho_{j+1}<\hat{\varrho}_j < \varrho_j$;
\vspace*{0.05cm}

b)\,\,$k_j - k_{j+1} = 2^{-j-2}\delta k \geq 2^{-j-3}k_j$
\vspace*{0.1cm}

\noindent We now set $B_j := B_{\varrho_j}(x_0)$ and $\hat{B}_j:= B_{\hat{\varrho}_j}(x_0)$ and we observe that
$$w_j := (u-k_j)_{-} \geq 2^{-j-3}k_j \chi_{\{u<k_{j+1}\}}.$$
Finally, we take a sequence of cut-off functions $\{\psi_j\}_j
\subseteq C^{\infty}_{0}(\hat{B}_{\varrho_j})$ such that
\begin{itemize}
\item $0\leq \psi_j\leq 1$ in $\hat{B}_j$;
\vspace*{0.05cm}
\item $\psi_j =1$ in $B_{j+1}$;
\vspace*{0.05cm}
\item $|\nabla \psi_j|\leq 2^{j+3}/r$.
\end{itemize}
We then choose $\psi = \psi_j$ and $w=w_j$ in \eqref{eq:7.18GK} and we inherit from \cite{GarainKinnunen} the estimates of $J_1$, $J_2$ and $J_3$, which we report here below for completeness:
\begin{equation} \label{eq:estimatesJi}
\begin{split}
 \mathrm{i)}&\,\,J_1, J_2 \leq c 2^{2j}k_j^2 r^{-2} |B_j \cap \{u<k_j\}| \\
 \mathrm{ii)}&\,\,J_3 \leq c 2^{j(n+2s)}k_j^2 r^{-2} |B_j \cap \{u<k_j\}|
 \end{split}
\end{equation}
On the other hand, recalling that $r\leq 1$, we also have
\begin{equation*}
\begin{aligned}
& \int_{B_{\varrho}(x_0)}a(x) uw\psi^2\, dx \leq \|a\|_{\infty}\int_{B_j}u w_j \psi_j^2 \,dx \\
& \qquad \leq  \|a\|_{\infty} \int_{B_j \cap \{u < k_j\}}k_j^2 \psi^2_j \, dx 
\leq \|a\|_{\infty}\cdot k_j^2 |B_j \cap \{u<k_j\}| \\
& \qquad \leq \|a\|_{\infty}\cdot r^{-2} k_j^2 |B_j \cap \{u<k_j\}|.
\end{aligned}
\end{equation*}
By combining this last estimate with \eqref{eq:estimatesJi}, we can then
follow the arguments of the proof of \cite[Lemma 7.1]{GarainKinnunen}, and we obtain
the claimed \eqref{eq:7.16GK}. 
\vspace*{0.1cm}

Finally, as a consequence of \eqref{eq:7.16GK} we obtain \eqref{eq:explansionConclusion}.
\end{proof}
\noindent With Lemma \ref{lem:expansion} at hand, we can prove
Proposition \ref{prop:weakHarnack}.
 \begin{proof}[Proof (of Proposition \ref{prop:weakHarnack}).] 
 The proof of \eqref{eq:8.1GK} can be obtained by
 arguing as in the proof
 of \cite[Lemma 4.1]{CastroKuusiPal}, by exploiting Lemma \ref{lem:expansion} in place
 of
 \cite[Lemma 3.2]{CastroKuusiPal}.
 
 We explicitly stress that the exponent $Q$ is independent of $\e\in(0,1]$, as the
 same is true of all the constants appearing in the previous lemmas.
\end{proof}
From Proposition \ref{prop:weakHarnack}, and using a classical covering argument, we obtain the following corollary.
\begin{corollary} \label{cor:BrezisNirenbergpernoi}
 Let $a\in L^\infty(\Omega),\,\text{$a\geq 0$ a.e.\,in $\Omega$}$. Moreover,
let $u$ be a weak superso\-lu\-tion of 
equation \eqref{eq:withZeroOrder} such that 
$u\geq 0$ a.e.\,in $\R^n\setminus\Omega$ and
$$\text{$u > 0$ a.e.\,on every open ball $B\Subset\Omega$}.$$
Then, for every open set $\Oo\Subset\Omega$ there exists $C = C(\Oo,u) > 0$ such that
$$\text{$u\geq C(\Oo,u) > 0$ a.e.\,in $\Oo$}.$$
\end{corollary}

\noindent\textbf{ii)\,\,The purely singular problem for $\LL$.}
The second auxiliary result we need concerns the 
\emph{un\-per\-tur\-bed}, purely singular version of problem \eqref{eq:Problem}$_\lambda$, that is,
\begin{equation} \label{eq:NoPerturbation} 
 \begin{cases}
 \LL u = \lambda u^{-\gamma} & \textrm{in $\Omega$},\\
 u>0 & \textrm{in $\Omega$},\\
 u= 0 & \textrm{in $\mathbb{R}^n \setminus \Omega$},
 \end{cases}
\end{equation}
In this context, we first prove the following proposition.
\begin{proposition}\label{prop:esistenzaGarain}
 Let $\e\in(0,1]$ be fixed. Moreover, let $\lambda > 0$ and $\gamma\in(0,1)$.
 
 Then, the following assertions hold.
 \begin{itemize}
  \item[i)] There exists a \emph{unique weak solution} $w_{\lambda,\e} 
  \in \mathcal{X}^{1,2}(\Omega)$
  of problem \eqref{eq:NoPerturbation}, which is the unique \emph{global minimizer of the functional}
  \begin{equation} \label{eq:funJlambdaSing}
   J_{\lambda,\e}(u) = 
  \frac{1}{2}\rho_\e(u)^2-\frac{\lambda}{1-\gamma}\int_\Omega|u|^{1-\gamma}\,dx.
  \end{equation}
  Moreover, $J_{\lambda,\e}(w_{\lambda,\e}) < 0$.
  \medskip
  
  \item[ii)]  $w_{\lambda,\e}\in L^\infty(\Omega)$, and there exists $c_1 > 0$ such that
  $$\|w_{\lambda,\e}\|_{L^\infty(\Omega)}\leq c_1\qquad\text{for every $\e \in(0,1]$}.$$
  \item[iii)] There exists $\e_0 > 0$
  with the following property:
  \emph{given any ball $B_{R}(x_0)\subseteq\Omega$ and any $0<r\leq \min\{1,R\}$, for every $0<\e\leq\e_0$
  we have}
  $$w_{\lambda,\e}\geq c_2\quad\text{a.e.\,on $B_r(x_0)$},$$
  for some constant $c_2 > 0$ independent of $\e$.
 \end{itemize}
 \end{proposition}
 We explicitly stress that, since problem \eqref{eq:NoPerturbation} is of the form
 \eqref{eq:pbgeneral} {(}with $f\equiv 0${)}, the definition of weak solution of \eqref{eq:NoPerturbation}
 is given in Definition \ref{def:weaksol}.
 \begin{proof}
 We prove separately the three assertions.
 \medskip
 
 \noindent i)\,\,The functional $J_{\lambda, \e}$ is weakly lower semicontinuous on $\mathcal{X}^{1,2}(\Omega)$ (for every $\e \in (0,1]$) and coercive on the same space, indeed, 
 \begin{equation*}
 J_{\lambda,\e}(u)\geq \dfrac{\rho_{\e}(u)^2}{2} - \dfrac{\lambda C_2}{1-\gamma}\|u\|^{1-\gamma}_{H^{1}_{0}(\Omega)} \geq  \dfrac{\|u\|^{2}_{H^{1}_{0}(\Omega)}}{2} - \dfrac{\lambda C}{1-\gamma}\|u\|^{1-\gamma}_{H^{1}_{0}(\Omega)},
 \end{equation*}
 \noindent where the positive constant $C$ depends only on $n$ and $\Omega$. Moreover, for $\delta>0$ small enough, and denoting with $e_{1,\e}$ the first Dirichlet eigenfunction of $\mathcal{L}_{\e}$, we also have
 \begin{equation*}
 J_{\lambda,\e}(\delta \, e_{1,\e}) < 0 = J_{\lambda,\e}(0),
 \end{equation*}
 \noindent hence $J_{\lambda,\e}$ has a global minimizer $w_{\lambda,\e}\not\equiv 0$.
 \medskip
 
 \noindent ii)\,\,The (global) boundedness of $w_{\lambda,\e}$ is proved
 in \cite[Lemma 2.2]{Garain}; in particular, since the proof in \cite{Garain}
 exploits the classical Stampacchia method and it is performed
 by getting rid of the nonlocal part of $\LL$, one can easily derive that
 the upper bound on $\|w_{\lambda_\e}\|_{L^\infty(\Omega)}$ is independent of $\e\in(0,1]$.
 In fact, we have
 $$0<w_{\lambda,\e}\leq 1+d,$$
 where the constant $d$ satisfies (for some $q > 2$)
 $$d^q = \lambda\,C|\{w_{\lambda,\e} > 1\}|^{q-2}2^{\frac{q(q-1)}{q-2}}\leq
 \lambda\,C|\Omega|^{q-2}2^{\frac{q(q-1)}{q-2}},$$
 and $C > 0$ only depends on $n,\gamma$ and on the measure of $\Omega$.
 \medskip
 
 \noindent iii)\,\,We split the proof of this assertion into three
 steps.
 \medskip 
 
 \textsc{Step 1).} In this first step, we prove that 
 \begin{equation} \label{eq:wlambdaeunfH01}
  \|w_{\lambda,\e}\|_{H^1_0(\Omega)}\leq \lambda\,c\quad\text{for all $\e\in(0,1]$}.
 \end{equation}
 for some constant $c > 0$ only depending on $\Omega$ and on $\gamma$.
 
  Indeed, since $w_{\lambda,\e}\in\mathcal{X}^{1,2}(\Omega)$ 
 is a weak solution
 of problem \eqref{eq:NoPerturbation}  (in the sense of
 De\-fi\-nition \ref{def:weaksol}), by choosing $w_{\lambda,\e}$ as a test
 function in \eqref{eq:weakformSol}, we get
  \begin{align*}
    \|w_{\lambda,\e}\|_{H^1_0(\Omega)}^2 &
    \leq \rho_\e(w_{\lambda,\e})^2
    = \frac{\lambda}{1-\gamma	}\int_{\Omega}w_{\lambda,\e}^{1-\gamma}\,dx \\
    & (\text{using H\"older's and Poincar\'e inequality}) \\   
    & \leq \frac{\lambda|\Omega|^{1-\frac{1-\gamma}{2}}}{1-\gamma}\|w_{\lambda,\e}\|^{1-\gamma}
    \leq \lambda\,c(\Omega,\gamma)\|w_{\lambda,\e}\|^{1-\gamma}_{H^1_0(\Omega)},
  \end{align*}
  where $c(\Omega,\gamma) > 0$ is a suitable constant only depending on $\Omega$ and on $\gamma$.
  From this, since we know that $w_{\lambda,\e}\not\equiv 0$, we immediately derive the
  claimed \eqref{eq:wlambdaeunfH01}.
  \medskip
  
  \textsc{Step 2).} In this second step we prove that, as $\e\to 0^+$, we have
  \begin{itemize}
   \item[a)] $w_{\lambda,\e}\to w_{\lambda}$ weakly in $\mathcal{X}^{1,2}(\Omega)$;
   \item[b)] $w_{\lambda,\e}\to w_{\lambda}$ strongly in $L^m(\Omega)$ for every $1\leq m < 2^*$;
  \end{itemize}
  where $w_{\lambda}\in H^1_0(\Omega)$ is the unique solution of the purely local problem
  \begin{equation} \label{eq:singularsololocale}
   \begin{cases}
 -\Delta u = \lambda u^{-\gamma} & \textrm{in $\Omega$},\\
 u>0 & \textrm{in $\Omega$},\\
 u= 0 & \textrm{in $\mathbb{R}^n \setminus \Omega$},
 \end{cases}
 \end{equation}
 To this end, we let $\{\e_j\}_j\subseteq(0,1]$ be any sequence converging to $0$ as $j\to+\infty$, and 
 we arbitrarily choose a subsequence
 $\{\e_{j_k}\}_k$  of $\{\e_j\}_j$. On account of
 \eqref{eq:wlambdaeunfH01}, and using
 Remark \ref{rem:spaceX12}\,-\,4), we know that there exists
 some function $\bar{w}\in\mathcal{X}^{1,2}(\Omega)$ such that, as $k\to+\infty$ and
 up to possibly 
 choosing a further subsequence,
 \begin{itemize}
   \item $w_{\lambda,\e_{j_k}}\to \bar{w}$ weakly in $\mathcal{X}^{1,2}(\Omega)$;
   \item $w_{\lambda,\e_{j_k}}\to \bar{w}$ strongly in $L^m(\Omega)$ for every $1\leq m < 2^*$;
  \end{itemize}
  from this, since $w_{\lambda,\e}$ is a global minimizer for $J_{\lambda,\e}$, we get
  \begin{align*}
   J_{\lambda}(\bar{w})
   & = \frac{1}{2}\|\bar{w}\|^2_{H^1_0(\Omega)}
   -\frac{\lambda}{1-\gamma}\int_\Omega |\bar{w}|^{1-\gamma}\,dx
   \\
   & \leq \liminf_{k\to+\infty}
   J_{\lambda,\e_{j_k}}(w_{\lambda,\e_{j_k}}) \\
   & \leq J_{\lambda}(\varphi)\quad\text{for every $\varphi\in H_0^1(\Omega)$},
  \end{align*}
  and thus $\bar{w}$ is a global minimizer for the functional $J_{\lambda}$
  naturally associated with the purely local, singular problem 
  \eqref{eq:singularsololocale}. As a consequence, since $J_{\lambda}$ as a
  unique global minimizer which is the unique solution $w_\lambda$ of \eqref{eq:singularsololocale},
  we derive that
  $$\bar{w} = w_{\lambda}.$$
  Due to arbitrariness of the sequence 
  $\{\e_j\}_j\subseteq(0,1]$ of its subsequence
 $\{\e_{j_k}\}_k$, and since the limit
 \emph{is always the same}, we conclude the validity of a)\,-\,b).
 \medskip
 
 \textsc{Step 3)}\,\,In this last step, we complete the proof of
 the assertion. First of all we observe that, since
 $w_{\lambda,\e}$ is a weak solution of problem \eqref{eq:NoPerturbation},
 we readily derive that $w_{\lambda,\e}$ is a also \emph{weak supersolution}
 (in the sense of Definition \ref{def:weaksubsolZeroOrder}) of
 $$\LL u = 0\quad\text{in $\Omega$};$$
 thus, by Proposition \ref{prop:weakHarnack} (and since $w_{\lambda,\e}\geq 0$ a.e.\,in $\R^n$) we get
 \begin{equation} \label{eq:HarnackwlambdaeI}
  w_{\lambda,\e}\geq c\,\left(\fint_{B_r(x_0)}w_{\lambda,\e}^Q \, dx\right)^{1/Q}\quad\text{a.e.\,on $B_r(x_0)$},
 \end{equation}
 where $c > 0$ does not depend on $\e$.
 On the other hand, since by \textsc{Step 2)} we know that $w_{\lambda,\e}\to w_\lambda$
 as $\e\to 0$ strongly in $L^m(\Omega)$ for every $1\leq m <2^*$, we have
 \begin{equation} \label{eq:convergenceMeanwlambdae}
  \lim_{\e\to 0^+}\fint_{B_r(x_0)}w_{\lambda,\e}^Q \, dx = \fint_{B_r(x_0)}w_{\lambda}^Q \, dx.
 \end{equation}
 Gathering \eqref{eq:HarnackwlambdaeI}-\eqref{eq:convergenceMeanwlambdae}, 
 and recalling that $w_\lambda$ is the unique solution of \eqref{eq:singularsololocale} (hence,
 $w_\lambda > 0$ a.e.\,in $\Omega$),
  we can then fined some $\e_0 > 0$ such that
  $$w_{\lambda,\e}\geq \frac{c}{2}\left(\fint_{B_r(x_0)}w_{\lambda}^Q \, dx\right)^{1/Q} > 0$$
  a.e.\,in $B_r(x_0)$ and \emph{for every $0<\e\leq \e_0$}. This ends the proof.
 \end{proof}
  We then conclude this paragraph (and the whole section) 
  with an apriori $L^\infty$-e\-sti\-mate for the weak solutions (provided they exist)
  of problem \eqref{eq:Problem}$_\lambda$.
 
\begin{theorem} \label{thm:uniformLinf}
Let $u\in \mathcal{X}^{1,2}(\Omega)$ be a weak solution of 
problem \eqref{eq:Problem}$_\lambda$
\emph{(}which implicitly depends on the fixed $\e\in(0,1]$\emph{)}, and suppose that
$$\text{$0<\lambda\leq\lambda_0$ for some $\lambda_0\geq 1$}.$$
Then, the following facts hold.
\begin{itemize}
 \item[i)] $u \in L^{\infty}(\Omega)$, and we have the apriori estimate
\begin{equation}\label{eq:Stima_dip_da_epsilon}
\|u\|_{L^{\infty}(\Omega)} \leq C \, \left(1 + \int_{\Omega}|u|^{2^{\ast} \beta_1}\, dx \right)^{\tfrac{1}{2^{\ast}(\beta_1 -1)}},
\end{equation}
where $\beta_{1}:= (2^{\ast}+1)/2$ and $C > 0$ only depends on $n,\lambda_0$ and $\Omega$.
\medskip

\item[ii)] There exists $r_1 \in(0,1)$ such that, if $\|u\|_{H^{1}_{0}(\Omega)}\leq r_1$,
then
\begin{equation}\label{eq:Stima_indip_da_epsilon}
\|u\|_{L^{\infty}(\Omega)} \leq C.
\end{equation}
for some constant $C = C(n,\lambda_0,\Omega) > 0$, but independent of $\e$.
\end{itemize} 
\end{theorem}
\begin{proof}
We prove the two assertions separately.
\medskip

i)\,\,We follow a classical Moser-type iteration scheme; for simplicity, we adopt the notation used in \cite[Theorem 1.1]{SVWZ2} in the case of a nonlinearity such that
$$|g(x,t)|\leq c(1+|t|^{2^{\ast}-1}).$$
We will highlight only the difference occurring due to the presence of the singular term $u^{-\gamma}$.\\
We first prove that $u\in L^{2^{\ast}\,\beta_{1}}(\Omega)$, where $\beta_{1}:=\tfrac{2^{\ast}+1}{2}$.\\
For $\beta>1$ and $T>0$ define the function
\begin{equation*}
\varphi(t):= \left\{ \begin{array}{rl}
-\beta T^{\beta -1}(t+T) + T^{\beta}, & t \leq -T\\
|t|^{\beta}, & -T<t <T\\
\beta T^{\beta -1}(t-T) + T^{\beta}, & t \geq T.
\end{array}\right.
\end{equation*}
Arguing similarly to \cite{SVWZ}, one can show that $\varphi \in \mathcal{X}^{1,2}(\Omega)$.
Repeating the first part of the argument of \cite[Lemma 3.2]{SVWZ2}, we find that
\begin{equation}\label{eq:3.5_Val2}
\begin{aligned}
& \|\varphi(u)\|^2_{L^{2^{\ast}}(\Omega)} \leq C(n,\Omega) \, \int_{\mathbb{R}^{n}}\varphi(u)\varphi'(u)\left(\lambda u^{-\gamma}+u^{2^{\ast}-1}\right)\, dx\\
&\qquad\leq C(n,\Omega)\Big(\lambda \beta \int_{\Omega}|u|^{2\beta -1-\gamma}\, dx +\beta \int_{\Omega}(\varphi(u))^2|u|^{2^{\ast}-2}\, dx\Big)\\
&\qquad\leq \lambda_0
C(n,\Omega)\beta\Big(1  +\int_{\Omega}|u|^{2\beta -1}\, dx + 
  \int_{\Omega}(\varphi(u))^2|u|^{2^{\ast}-2}\, dx\Big),
\end{aligned}
\end{equation}
\noindent which is analogous to \cite[Equation (3.5)]{SVWZ2}. Now, for every $R>0$
\begin{equation}\label{eq:3.6_Val2}
\begin{aligned}
& \int_{\Omega}(\varphi(u))^2|u|^{2^{\ast}-2}\, dx \\
& \qquad = \int_{\{|u|\leq R\}}(\varphi(u))^2|u|^{2^{\ast}-2}\, dx + \int_{\{|u|> R\}}(\varphi(u))^2|u|^{2^{\ast}-2}\, dx\\
&\qquad \leq R^{2^*-1}\int_{\{|u|\leq R\}}\dfrac{(\varphi(u))^2}{|u|}\, dx \\
& \qquad\qquad+ \left(\int_{\Omega}(\varphi(u))^{2^{\ast}}\, dx \right)^{2/2^{\ast}} \left(\int_{\{|u|>R\}}|u|^{2^{\ast}}\, dx\right)^{(2^{\ast}-2)/2^{\ast}}
\end{aligned}
\end{equation}
\noindent and one can follow verbatim the proof of \cite[Lemma 3.2]{SVWZ2} to get $u\in L^{2^{\ast}\,\beta_{1}}(\Omega)$ after choosing $R$ properly and letting $T\to +\infty$. From here the proof is exactly the same in \cite[Theorem 1.1]{SVWZ2}. It is easy to recognize that the upper bound found in \eqref{eq:Stima_dip_da_epsilon} depends on $\e$ in a non explicit way (due to the choice of $R$).
\medskip

ii)\,\,Assume now that $\|u\|_{H^{1}_{0}(\Omega)}\leq r_1$ for some
$r_1 >0$ (to be chosen conveniently small in a moment), 
and let us go back to \eqref{eq:3.6_Val2}: choosing $R = 1$, we get
\begin{equation}\label{eq:3.6_Val2_con_R=1}
\begin{aligned}
&\int_{\Omega}(\varphi(u))^2|u|^{2^{\ast}-2}\, dx \\
& \qquad
 = \int_{\{|u|\leq 1\}}(\varphi(u))^2|u|^{2^{\ast}-2}\, dx + \int_{\{|u|> 1\}}(\varphi(u))^2|u|^{2^{\ast}-2}\, dx\\
&\qquad\leq \int_{\{|u|\leq 1\}}\dfrac{(\varphi(u))^2}{|u|}\,dx + 
\|\varphi(u)\|^2_{L^{2^*}(\Omega)}\|u\|^{2^*-2}_{L^{2^*}(\Omega)} \\
& \qquad\leq \int_{\{|u|\leq 1\}}\dfrac{(\varphi(u))^2}{|u|}\, dx + 
S_n^{2-2^*}r_1\|\varphi(u)\|^2_{L^{2^*}(\Omega)},
\end{aligned}
\end{equation}
where $S_n > 0$ is the best Sobolev constant. Now, choosing $r_1>0$ such that 
\begin{equation*}
\lambda_0C(n,\Omega)S_n^{2-2^*}\beta r_1 <\dfrac{1}{2},
\end{equation*}
\noindent we can reabsorb the last term of \eqref{eq:3.6_Val2_con_R=1} on the left hand side of \eqref{eq:3.5_Val2}, and after passing to the limit as $T\to +\infty$ (as in \cite{SVWZ2}) we get
\begin{align*}
\left(\int_{\Omega}|u|^{2^{\ast}\beta_1}\, dx\right)^{2/2^{\ast}} &
 \leq 4\lambda_0\hat{C}(n,\Omega)\beta\Big(1+\int_{\Omega}|u|^{2^{\ast}}\, dx\Big) \\
 & \leq 4\lambda_0\hat{C}(n,\Omega)\beta \, r_1^{2^*} \leq C(n,\Omega,\beta,\lambda_0)<+\infty.
\end{align*}
From this, one can perform the same iterative argument used in 
the proof of \cite[Theorem 1.1]{SVWZ2}, and the demonstration is complete.
\end{proof}
\section{Proof of Theorem \ref{thm:main} } \label{sec:proofThm}
Thanks to all the results established so far, we are finally in a position to provide
the full proof of Theorem \ref{thm:main}. In doing this, we mainly
follow the approach in \cite{Haitao}; moreover, in order to keep the exposition as clear as possible,
we split such a proof into several independent results. 
\medskip

\noindent\textbf{1) Existence of a first solution.} To begin with, we define
 \begin{equation}\label{eq:DefinitionLambda}
 \Lambda_{\e} := \sup \{ \lambda >0: \eqref{eq:Problem}_\lambda \textrm{ admits a weak solution}\}.
 \end{equation}
 We then turn to prove in this first part of the section the following facts:
 \vspace*{0.1cm}
 
 a)\,\,$\Lambda_{\e}$ is well-defined and $\Lambda < +\infty$;
 \vspace*{0.05cm}
 
 b)\,\,problem \eqref{eq:Problem}$_\lambda$ admits a weak solution for every $0<\lambda\leq \Lambda_{\e}$.
 \medskip
 
 \noindent Troughout what follows, we denote by $I_{\lambda,\e}$ the functional
 (depending on $\lambda$ and $\e$) naturally associated with
 problem \eqref{eq:Problem}$_\lambda$, that is,
 \begin{equation} \label{eq:functionalIlambdae}
I_{\lambda,\e}(u) := \dfrac{1}{2}\rho_{\e}(u)^2 - \dfrac{\lambda}{1-\gamma}\int_{\Omega}|u|^{1-\gamma}\, dx - \dfrac{1}{2^*}\int_{\Omega}|u|^{2^*}\, dx, \quad u \in \mathcal{X}^{1,2}(\Omega).
\end{equation}
On the other hand, we will avoid to keep tracking of the dependence on $\e$
of any possible of $\LL$, unless it is strictly needed (for instance,
if we need to choose $\e$ conveniently small, or if we need some estimates
uniform in $\e$).
\medskip

 We begin by proving assertion a).
 \begin{lemma} \label{lem:Lambdafinito}
  The following assertions hold.
  \begin{itemize}
   \item[1)] There exists $r_0 > 0$, only depending on the dimension $n$, with the following property:
   \emph{for every $0<r\leq r_0$ there exists $\lambda_* > 0$, only depending on $r,\gamma$ and on
   the measure of $\Omega$, such that problem
   \eqref{eq:Problem}$_\lambda$
 possesses at least one weak solution $u_{\lambda,\e}\in\mathcal{X}^{1,2}(\Omega)$
 for every $0<\lambda\leq\lambda_*$,
 further satisfying}
 \vspace*{0.1cm}
 
 \begin{itemize}
  \item[{i)}] $\|u_{\lambda,\e}\|_{H_0^1(\Omega)}\leq r$;
  \item[{ii)}] $u_{\lambda,\e}$ is a local minimizer of the functional $I_{\lambda,\e}$
  in \eqref{eq:functionalIlambdae}.
\end{itemize}
Hence, in particular, $\Lambda_\e\geq \lambda_* > 0$ for every $\e\in(0,1]$.
\vspace*{0.1cm}

   \item[2)] There exists $\Lambda_* > 0$, only depending on the measure of $\Omega$, such that
   $$\Lambda_\e\leq \Lambda_*\quad\text{for every $\e\in(0,1]$}.$$
  \end{itemize}
 \end{lemma}
 \begin{proof}
 We prove the two assertions separately.
 \medskip
 
 1)\,\,First of all we observe that, 
by the (local) Sobolev inequality and the H\"older inequality, 
for every $u\in\mathcal{X}^{1,2}(\Omega)$ we have the following estimates
\begin{equation} \label{eq:todeduce21Haitao}
\begin{split}
& \mathrm{a)}\,\,\int_\Omega |u|^{2^*}\,dx \leq C_1\|u\|^{2^*}_{H^1_0(\Omega)}; \\
& \mathrm{b)}\,\,\int_{\Omega}|u|^{1-\gamma}\, dx \leq 
|\Omega|^{1-(1-\gamma)/2^{*}}\|u\|_{L^{2^*}(\Omega)}^{(1-\gamma)/2^{*}} \leq C_2\|u\|^{1-\gamma}_{H^1_0(\Omega)};
\end{split}
\end{equation}
where $C_1 > 0$ is a constant only depending on $n$, while 
$C_2 > 0$ depends on $n,\gamma$ and on the
measure of $\Omega$; as a consequence, setting
$$B_{r}:= \{ u \in \mathcal{X}^{1,2}(\Omega): \|u\|_{H^1_0(\Omega)}\leq r\}$$
(and reminding that $\rho_{\e}(\cdot)\geq \|\cdot\|_{H^1_0(\Omega)}$,
see \eqref{eq:equivalenceuniforme}), it is possible to find
some number $r_0 > 0$, only depending on the dimension $n$,
such that
\begin{equation}\label{eq:2.1Haitao}
\left\{ \begin{array}{rl}
\tfrac{1}{2}\rho_{\e}(u)^2 - \tfrac{1}{2^*}\|u\|^{2^*}_{L^{2^*}(\Omega)}\geq r^2/4 & \textrm{for all } u \in \partial B_{r},\\[0.2cm]
\tfrac{1}{2}\rho_{\e}(u)^2 - \tfrac{1}{2^*}\|u\|^{2^*}_{L^{2^*}(\Omega)}\geq 0 & \textrm{for all } u \in B_{r}.
\end{array}\right.
\end{equation}
for every $0<r\leq r_0$. Then, if $r\in(0,r_0]$ is fixed,
by \eqref{eq:2.1Haitao} and
 \eqref{eq:todeduce21Haitao}\,-\,b) there exists some $\lambda_{\star}>0$,
 depending on $r$ and on $|\Omega|$ but \emph{independent of $\e$},
such that
\begin{equation}\label{eq:2.2Haitao}
\begin{split}
I_{\lambda,\e}\big|_{\partial B_{r_0}} & \geq \frac{r^2}{4} - \frac{\lambda\,C_2}{1-\gamma}r^{1-\gamma}
\geq \frac{r^2}{8} \quad \textrm{for every } \lambda \in (0,\lambda_{\star}],
\end{split}
\end{equation}
where $I_{\lambda,\e}$ is as in \eqref{eq:functionalIlambdae}.
For every fixed  $\lambda\in(0,\lambda_*]$, we now set 
$$c = c_{\lambda,\e} = \inf_{B_{r}}I_{\lambda,\e},$$ 
and we notice that $c<0$. Indeed, for every $v\not\equiv 0$, it holds that
\begin{equation*}
I_{\lambda,\e}(tv) = t^2 \rho_{\e}(v)^2 - \dfrac{\lambda_{\star}}{1-\gamma}t^{1-\gamma}\int_{\Omega}|v|^{1-\gamma}\, dx - \dfrac{t^{2^{*}}}{2^*}\int_{\Omega}|v|^{2^*}\, dx,
\end{equation*}
\noindent which becomes negative for $t>0$ small enough, since $0<1-\gamma<1$. The argument is now pretty standard. We first consider a minimizing sequence 
$\{u_j\}_j
\subseteq B_{r}$ re\-la\-ted to $c$ and we notice that,
since $\{u_j\}_j$ is bounded in $\mathcal{X}^{1,2}(\Omega)$
(see \eqref{eq:equivalenceuniforme}),
 it is possible to find $u_{\lambda,\e}\in\mathcal{X}^{1,2}(\Omega)$ such that, up to subsequences,
\begin{itemize}
\item $u_j \to u_{\lambda,\e}$ as $j \to +\infty$ weakly in $\mathcal{X}^{1,2}(\Omega)$;
\item $u_j \to u_{\lambda,\e}$ as $j \to +\infty$ strongly in $L^{m}(\Omega)$ for every $m \in [1,2^*)$;
\item $u_j \to u_{\lambda,\e}$ as $j \to +\infty$ pointwise a.e. in $\Omega$.
\end{itemize}
Moreover, since $I_{\lambda,\e}(|u|) \leq I_{\lambda,\e}(u)$, we may also assume that $u_j \geq 0$.
Thus, combining \eqref{eq:2.2Haitao} with the fact that $c_{\lambda,\e}<0$, 
we realize that there exists a constant $\vartheta_0>0$,
 independent of $j$,   such that
\begin{equation} \label{eq:boundrhoujtodeduce}
\|u_j\|_{H_0^1(\Omega)}\leq r - \vartheta_0.
\end{equation}
Now, by combining the
algebraic inequality $(a+b)^p\leq a^b+b^p$ (holding true for
all $a,b\geq 0$ and $0<p<1$) with the
H\"{o}lder inequality, as $j\to +\infty$ we have
\begin{equation*}
\begin{split}
\int_{\Omega}u_j^{1-\gamma} & \leq \int_{\Omega}u_{\lambda,\e}^{1-\gamma}+
\int_{\Omega}|u_j - u_{\lambda,\e}|^{1-\gamma}\,dx 
\\
& \leq \int_{\Omega}u_{\lambda,\e}^{1-\gamma}+ 
 C \|u_j-u_{\lambda,\e}\|^{1-\gamma}_{L^{2}(\Omega)} = \int_{\Omega}u_{\lambda,\e}^{1-\gamma} + o(1),
\end{split}
\end{equation*}
\noindent and similarly
\begin{equation*}
\int_{\Omega}u_{\lambda,\e}^{1-\gamma} \leq \int_{\Omega}u_{j}^{1-\gamma}+\int_{\Omega}|u_j - 
u_{\lambda,\e}|^{1-\gamma} = \int_{\Omega}u_{j}^{1-\gamma} + o(1),
\end{equation*}
\noindent which in turn implies
\begin{equation}\label{eq:2.3Haitao}
\int_{\Omega}u_{j}^{1-\gamma} \, dx = \int_{\Omega}u_{\lambda,\e}^{1-\gamma}\, dx + o(1), 
\quad \textrm{as } j \to +\infty.
\end{equation}
By Brezis-Lieb lemma, see \cite{BrezisLieb}, it is now well-known that
\begin{equation*}
\|u_j\|^{2^*}_{L^{2^*}(\Omega)} = \|u_{\lambda,\e}\|^{2^*}_{L^{2^*}(\Omega)} + \|u_j - u_{\lambda,\e}\|^{2^*}_{L^{2^*}(\Omega)} +o(1), \quad \textrm{as } j \to +\infty;
\end{equation*}
moreover, since $u_j\to u_{\lambda,\e}$ weakly in $\mathcal{X}^{1,2}(\Omega)$
(and recalling that $\rho_\e(\cdot)$ and $\rho(\cdot)$ 
are \emph{uniformly equivalent} on $\mathcal{X}^{1,2}(\Omega)$, see 
\eqref{eq:equivalenceuniforme}), we have
\begin{equation}\label{eq:2.5Haitao}
\rho_{\e}(u_j)^2 = \rho(u_{\lambda,\e})^2 + \rho_{\e}(u_j -u_{\lambda,\e})^2 + o(1), 
\quad \textrm{as } j \to +\infty.
\end{equation}
By combining \eqref{eq:2.5Haitao} and \eqref{eq:boundrhoujtodeduce}, 
it follows that $u_{\lambda,\e}\in B_{r}$ and that
$u_j - u_{\lambda,\e} \in B_{r}$ for big enough $j$, and this allows to use the second line of \eqref{eq:2.1Haitao} on $u_j - u_{\lambda,\e}$. 
Using now \eqref{eq:2.3Haitao}-\eqref{eq:2.5Haitao}, as $j\to +\infty$, we find
\begin{equation*}
\begin{aligned}
c_{\lambda,\e} &= I_{\lambda,\e}(u_j) + o(1)\\
&= I_{\lambda,\e}(u_{\star}) + \dfrac{1}{2}\rho(u_j - u_{\lambda,\e})^2 - \dfrac{1}{2^*}\|u_j - u_{\lambda,\e}\|^{2^*}_{L^{2^*}(\Omega)}+ o(1)\\
&\geq I_{\lambda,\e}(u_{\lambda,\e}) + o(1) \geq c_{\lambda,\e} + o(1),
\end{aligned}
\end{equation*}
\noindent which proves that $u_{\lambda,\e} \geq 0$, $u_{\lambda,\e}\not\equiv
0$ is a local minimizer of $I_{\lambda,\e}$ in the $\mathcal{X}^{1,2}(\Omega)$-to\-po\-logy.
From this, using the Strong Maximum Principle in \cite[Theorem 3.1]{BMV}, and arguing as in the proof of \cite[Lemma 2.1]{Haitao}, we conclude 
that $u_{\lambda,\e}$ is actually a weak solution of  \eqref{eq:Problem}$_\lambda$,
further satisfying i)\,-\,ii).
\medskip

2)\,\,Following \cite{Haitao}, we consider 
the \emph{first 
Dirichlet eigenvalue $\mu_{1,\e}$ of $\LL$} in $\Omega$, 
which is defined (through the usual variational formulation) as follows
$$\mu_{1,\e} = \min\big\{\rho_{\e}(u)^2:\,\text{$u\in\mathcal{X}^{1,2}(\Omega)$ and $\|u\|^2_{L^2(\Omega)} = 1$}\big\},$$
and we let $e_{1,\e}\in\mathcal{X}^{1,2}(\Omega)$ 
be the principal eigenfunction associated with $\mu_{1,\e}$, i.e.,
\vspace*{0.1cm}

a)\,\,$\|e_{1,\e}\|_{L^2(\Omega)} = 1$ and $e_{1,\e} > 0$ a.e.\,in $\Omega$; 
\vspace*{0.05cm}

b)\,\,$\rho_{\e}(e_{1,\e})^2 = \mu_1$.
\vspace*{0.1cm}

\noindent We refer, e.g., to \cite{BDVV3} for the proof of the existence of $e_1$.
We explicitly observe that, since $\rho_\e(\cdot)\geq \|\cdot\|_{H_0^1(\Omega)}$, we clearly have
$$\mu_{1,\e}\geq \hat{\mu}_1 = 
\min\big\{\|u\|_{H^1_0(\Omega)}^2:\,\text{$u\in\mathcal{X}^{1,2}(\Omega)$ 
and $\|u\|^2_{L^2(\Omega)} = 1$}\big\}$$
(note that $\hat{\mu}_1$ is the first Dirichlet eigenvalue of $-\Delta$ in $\Omega$).

Now, assuming that there exists
a weak solution $u\in\mathcal{X}^{1,2}(\Omega)$ of problem
\eqref{eq:Problem}$_\lambda$ (for some $\lambda > 0$),
and using this $e_{1,\e}$ as a test function in \eqref{eq:weakformSol}, we get
\begin{equation*}
\begin{split}
\mu_{1,0}\int_{\Omega}ue_{1,\e} \, dx
 & \geq \mu_{1\,e} \int_{\Omega}ue_{1,\e} \, dx = \mathcal{B}_{{\e}}(u,e_{1,\e}) \\
 & = \int_{\Omega}(\lambda u^{-\gamma} + u^{2^{*}-1})e_{1,\e} \, dx;
 \end{split}
\end{equation*}
thus, choosing $\Lambda_* > 0$ in such a way that
$$\Lambda_*t^{-\gamma} + t^{2^{*}-1} > 2 \mu_{1,0} t, \quad \textrm{for every } t>0$$
(notice that this choice of $\Lambda_* > 0$ only depends on $\mu_{1,0}$, and  it
is independent of $\e$),
we conclude that $\lambda < {\Lambda}_*$, and then $\Lambda < \Lambda_* <+\infty$, as desired.
 \end{proof}
%
 Now we have established Lemma \ref{lem:Lambdafinito}, we then turn to 
 prove assertion b), namely
 the existence of at least one weak solution of problem \eqref{eq:Problem}$_\lambda$ for
 every $0<\lambda\leq \Lambda_{\e}$. To this end, following \cite{Haitao}
 we first establish some  preliminary results.
 Throughout what follows, we tacitly understand that 
 $\e\in(0,1]$ is arbitrary \emph{but fixed}.
 \medskip
 
 To begin with, we prove the following simple yet important technical lemma.
 \begin{lemma} \label{lem:SMPsoprasotto}
  Let $w,u\in\mathcal{X}^{1,2}(\Omega)$ be a \emph{weak subsolution} 
  \emph{[}resp.\,\emph{weak supersolution}\emph{]}
  and a \emph{weak solution} of pro\-blem \eqref{eq:Problem}$_\lambda$, respectively.
  We assume that
  \begin{itemize}
   \item[a)] $w\leq u$ \emph{[}resp.\,$w\geq u$\emph{]} a.e.\,in $\Omega$;
   \item[b)] for every open set $\Oo\Subset\Omega$ there exists $C = C(\Oo,w) > 0$ such that
   $$\text{$w \geq C$ a.e.\,in $\Oo$}.$$
  \end{itemize}
  Then, either $w\equiv u$ or $w<u$ \emph{[}resp.\,$w > u$\emph{]} a.e.\,in $\Omega$.
 \end{lemma}
 \begin{proof}
  We limit ourselves to consider only the case when $w$ is a \emph{weak subsolution} of 
  pro\-blem \eqref{eq:Problem}$_\lambda$, since the case when $w$ is a
  weak supersolution is analogous.
  
  To begin with, we arbitrarily fix a bounded open set $\Oo\Subset\Omega$ and we observe that, since $w$ 
  is a weak subsolution of problem \eqref{eq:Problem}$_\lambda$ and since $u$
  is a weak solution of the same problem, we have the following computations:
   \begin{align*}
  \LL (u-w) & \geq \lambda(u^{-\gamma}-w^{-\gamma})
  +(u^{2^*-1}-w^{2^*-1}) \\
  & (\text{since $w\leq u$, see assumption a)}) \\
  & \geq \lambda({u}^{-\gamma}-w^{-\gamma}) \\
  & (\text{by the Mean Value Theorem, for some $\theta\in(0,1)$}) \\
  & = -\gamma\lambda(\theta {u}+(1-\theta)w)^{-\gamma-1}(u-w) \\
  & \geq -\gamma\lambda w^{-\gamma-1}(u-w) \\
  & (\text{by assumption b)}) \\
  & \geq -\gamma\lambda C^{-\gamma-1}(u-w),
 \end{align*}
 in the weak sense on $\Oo$
 (here, $C > 0$ is a constant depending on $\Oo$ and on $w$). 
 
 As a consequence of this fact,
 and since $w\leq u$ a.e.\,in $\Omega$ (hence, in $\R^n$), we are then entitled
 to apply the Strong Maximum Principle in \cite[Theorem 3.1]{BMV} to the
 function $v = u-w\in\mathcal{X}_+^{1,2}(\Omega)$ (with $f \equiv \gamma\lambda C^{-\gamma-1}t$),
 obtaining that
 $$\text{either $v\equiv 0$ or $v > 0$ a.e.\,in $\Oo$}.$$
 Due to the arbitrariness of $\Oo\Subset\Omega$, this completes the proof.
 \end{proof}
 \begin{remark} \label{rem:assumptionbnonserve}
  We explicitly observe that, if $w\in\mathcal{X}^{1,2}(\Omega)$ is a weak \emph{supersolution}
  of problem \eqref{eq:Problem}$_\lambda$, it follows from Remark \ref{rem:defweaksolPb}-3) that
  assumption b) in Lemma \ref{lem:SMPsoprasotto} is \emph{always satisfied}. 
  Hence, if $u\in\mathcal{X}^{1,2}(\Omega)$ is a weak \emph{solution} of \eqref{eq:Problem}$_\lambda$, we get
  $$(\text{$u\leq w$ a.e.\,in $\Omega$})\,\,\Longrightarrow\,\,(
  \text{either $u\equiv w$ or $u < w$ a.e.\,in $\Omega$}).$$
 \end{remark}
 We now turn to state the following Perron-type lemma, with extends \cite[Lemma 2.1]{Garain} to the 
 case of \emph{critical nonlinearities} 
 and \cite[Lemma 2.2]{Haitao} to the case of mixed local-nonlocal operators.
\begin{lemma}\label{lem:2.2Haitao}
Let $\underline{u}, \overline{u}\in \mathcal{X}^{1,2}(\Omega)$ be a weak subsolution and a weak supersolution, respectively, of problem \eqref{eq:Problem}$_\lambda$. 
We assume that 
\begin{itemize}
\item[a)] $\underline{u}(x) \leq \overline{u}(x)$ for a.e.\,$x\in \Omega$;
\item[b)] for every open set $\Oo\Subset\Omega$ there exists
$C = C(\Oo,\underline{u}) > 0$ such that
$$\text{$\underline{u}\geq C$ a.e.\,in $\Oo$}.$$
\end{itemize}
Then, there exists a weak solution $u \in \mathcal{X}^{1,2}(\Omega)$ of \eqref{eq:Problem} such that 
$$\text{$\underline{u}(x) \leq u(x) \leq \overline{u}(x)$ for a.e. $x \in \Omega$}.$$
\end{lemma}
\begin{proof}
The proof of this lemma is totally analogous to that of
of \cite[Lemma 2.2]{Haitao}, where the purely local counterpart of 
problem \eqref{eq:Problem}$_\lambda$ is considered; in order to treat the (only)
extra term arising from the nonlocal part of $\LL$ one may use
the estimates in \cite{GiacMukSre} (where the pure nonlocal counterpart of 
problem \eqref{eq:Problem}$_\lambda$ is considered).
\end{proof}
We are now ready to prove the existence of a weak solution of \eqref{eq:Problem}$_\lambda$. 
It is an adaptation to our setting of \cite[Lemma 2.3]{Haitao}.
\begin{lemma}\label{lem:2.3Haitao}
Problem \eqref{eq:Problem}$_\lambda$ admits \emph{(}at least\emph{)}
one weak solution $u_{\lambda}\in \mathcal{X}^{1,2}(\Omega)$ for every fixed
$\lambda \in (0,\Lambda_\e]$.
\end{lemma}
\begin{proof}
We first assume that $0<\lambda<\Lambda_\e$. In this case,
the idea of the proof is rather standard: we want to construct both a weak subsolution and a weak supersolution and then apply Lemma \ref{lem:2.2Haitao}. 

Let us start with the weak subsolution. 
By Proposition \ref{prop:esistenzaGarain}, we know that for every $\lambda \in (0,\Lambda_{\e})$ (and actually for all $\lambda >0$) there exists a unique solution $w_{\lambda,\e}$ of the {\it purely singular} problem \eqref{eq:NoPerturbation}, which is the Euler-Lagrange equation naturally associated with the functional
\begin{equation*}
J_{\lambda,\e}(u):= \dfrac{1}{2}\, \rho_{\e}(u)^2 - \dfrac{\lambda}{1-\gamma} \int_{\Omega}|u|^{1-\gamma}\, dx, \quad u \in \mathcal{X}^{1,2}(\Omega).
\end{equation*}

Then, it is readily seen that $w_{\lambda}$ is a weak subsolution of \eqref{eq:Problem}$_\lambda$. 

Let us now look for a weak supersolution. By the very definition of $\Lambda_\e$, we know that there necessarily exists $\lambda' \in (\lambda, \Lambda_\e)$ such that \eqref{eq:Problem}$_{\lambda'}$ admits a weak solution $u_{\lambda'}$, and this can be easily taken as a weak supersolution of \eqref{eq:Problem}$_{\lambda}$.\\
We now claim that
\begin{equation}\label{eq:Claim}
w_{\lambda}(x) \leq u_{\lambda'}(x), \quad \textrm{for a.e. } x\in \Omega.
\end{equation}
To this aim, let us consider a smooth non-decreasing function $\theta: \mathbb{R}\to \mathbb{R}$ such that
$$\theta(t) =1 \, \textrm{ for } t \geq 1 \quad \textrm{ and } \quad \theta(t) = 0 \, \textrm{ for } t \leq 0,$$
\noindent and it is linked in a smooth way for $t \in (0,1)$. We further define the function
$$\theta_{\sigma}(t):= \theta \left(\dfrac{t}{\sigma}\right) \quad 
(\text{$\sigma>0$ and $t\in \mathbb{R}$}).$$
Due to its definition, we are entitled to use the function $\theta_\sigma(w_{\lambda}- u_{\lambda'})$ as a test function in both \eqref{eq:Problem}$_{\lambda'}$ (solved by $u_{\lambda'}$) and \eqref{eq:NoPerturbation} (solved by $w_{\lambda}$). Thus, we have
\begin{equation}\label{eq:solvedBywlambda}
\begin{aligned}
&\int_{\Omega}\nabla w_{\lambda}\cdot \nabla(w_{\lambda}-u_{\lambda'})\theta'_{\sigma}
(w_{\lambda}-u_{\lambda'})\, dx \\
&+ \iint_{\mathbb{R}^{2n}}\dfrac{(w_{\lambda}(x)-w_{\lambda}(y))(\theta_{\sigma}(w_{\lambda}(x)-u_{\lambda'}(x)) - \theta_{\sigma}(w_{\lambda}(y)-u_{\lambda'}(y)))}{|x-y|^{n+2s}}\, dxdy\\
&\qquad - \lambda \int_{\Omega}\dfrac{\theta_{\sigma}(w_{\lambda}-u_{\lambda'})}{w_{\lambda}^{\gamma}}\, dx = 0,
\end{aligned}
\end{equation}
\noindent and
\begin{equation}\label{eq:solvedByuOverlinelambda}
\begin{aligned}
&\int_{\Omega}\nabla u_{\lambda'}\cdot \nabla(w_{\lambda}-u_{\lambda'})\theta'_{\sigma}(w_{\lambda}-u_{\lambda'})\, dx \\
&+ \iint_{\mathbb{R}^{2n}}\dfrac{(u_{\lambda'}(x)-u_{\lambda'}(y))(\theta_{\sigma}(w_{\lambda}(x)-u_{\lambda'}(x)) - \theta_{\sigma}(w_{\lambda}(y)-u_{\lambda'}(y)))}{|x-y|^{n+2s}}\, dxdy\\
& - \lambda' \int_{\Omega}\dfrac{\theta_{\sigma}(w_{\lambda}-u_{\lambda'})}{u_{\lambda'}^{\gamma}}\, dx  - \int_{\Omega}u_{\lambda'}^{2^{*}-1}\theta_{\sigma}(w_{\lambda}-u_{\lambda'})\, dx= 0.
\end{aligned}
\end{equation}
Subtracting \eqref{eq:solvedBywlambda} from  \eqref{eq:solvedByuOverlinelambda} we get
\begin{equation}\label{eq:Subtraction}
\begin{aligned}
0 &\geq  - \int_{\Omega}|\nabla(u_{\lambda'}-w_{\lambda})|^2 \theta'_{\sigma}(w_{\lambda}-u_{\lambda'})\, dx\\
&\qquad- \iint_{\mathbb{R}^{2n}}\frac{1}{|x-y|^{n+2s}}\Big\{
 \big(u_{\lambda'}(x)-u_{\lambda'}(y)-w_{\lambda}(x)+w_{\lambda}(y)\big)\times \\
& \qquad\qquad\qquad \times \big(\theta_{\sigma}(w_{\lambda}(x)-u_{\lambda'}(x))-\theta_{\sigma}(w_{\lambda}(y)-u_{\lambda'}(y))\big)\Big\}\, dxdy\\
&= \int_{\Omega}\left( \dfrac{\lambda'}{u_{\lambda'}^{\gamma}} - \dfrac{\lambda}{w_{\lambda}^{\gamma}} + u_{\lambda'}^{2^{*}-1}\right) \theta_{\sigma}(w_{\lambda}-u_{\lambda'})\, dx \\
&\geq \lambda \int_{\Omega}\left( \dfrac{1}{u_{\lambda'}^{\gamma}}- \dfrac{1}{w_{\lambda}^{\gamma}}\right)\theta_{\sigma}(w_{\lambda}-u_{\lambda'})\, dx.
\end{aligned}
\end{equation}
Let us justify the first inequality in \eqref{eq:Subtraction}. The local part is clearly non-positive. Regarding the nonlocal-part, 
if $(x,y)$ are such that
$$w_{\lambda}(x)- u_{\lambda'}(x) \geq w_{\lambda}(y)- u_{\lambda'}(y),$$
\noindent then
\begin{equation*}
\theta_{\sigma}(w_{\lambda}(x)-u_{\lambda'}(x))-\theta_{\sigma}(w_{\lambda}(y)-u_{\lambda'}(y)) \geq 0,
\end{equation*}
\noindent while, if $(x,y)$ are such that
$$w_{\lambda}(x)- u_{\lambda'}(x) \leq w_{\lambda}(y)- u_{\lambda'}(y),$$
\noindent then
\begin{equation*}
\theta_{\sigma}(w_{\lambda}(x)-u_{\lambda'}(x))-\theta_{\sigma}(w_{\lambda}(y)-u_{\lambda'}(y)) \leq 0.
\end{equation*}
Coming back to \eqref{eq:Subtraction}, letting $\sigma \to 0^+$ we find that
$$\int_{\{w_{\lambda}> u_{\lambda'}\}}\left( \dfrac{1}{u_{\lambda'}^{\gamma}}- \dfrac{1}{w_{\lambda}^{\gamma}}\right)\, dx \leq 0,$$
\noindent and this implies that
$$\left|\left\{ x \in \Omega: w_{\lambda}(x) > u_{\lambda'}(x)\right\}\right| = 0,$$
\noindent as claimed in \eqref{eq:Claim}. \vspace*{0.1cm}

With \eqref{eq:Claim} at hand, we are ready to complete the proof
of the lemma: in fact, setting  $\overline{u}= u_{\lambda'}$ and $\underline{u} = w_{\lambda}$, 
by \eqref{eq:Claim} and Proposition \ref{prop:esistenzaGarain} we know that
\vspace*{0.1cm}

i)\,\,$\underline{u}$ is weak subsolution and $\overline{u}$ is a weak supersolution
of problem \eqref{eq:Problem}$_\lambda$;
\vspace*{0.05cm}

ii)\,\,$\underline{u}$ and $\overline{u}$ satisfy assumptions a)-b) in Lemma \ref{lem:2.2Haitao}.
\vspace*{0.1cm}

\noindent We can then apply Lemma \ref{lem:2.2Haitao}, which therefore proving that problem \eqref{eq:Problem}$_{\lambda}$ admits a weak solution $u_{\lambda}$ for every $\lambda \in (0,\Lambda_\e)$. Moreover,
by Proposition \ref{prop:esistenzaGarain} we have
$$I_{\lambda,\e}(u_{\lambda}) \leq I_{\lambda,\e}(w_{\lambda}) \leq J_{\lambda,\e}(w_{\lambda,\e}) <0.$$
For $\lambda = \Lambda$ it is now sufficient to repeat the last part of the proof of \cite[Lemma 2.3]{Haitao} with the obvious modifications.
\end{proof}
\medskip

\begin{lemma} \label{lem:25Haitao}
  Let $\underline{u},\overline{u},u_\lambda
   \in\mathcal{X}^{1,2}(\Omega)$ be, respectively, the weak subsolution, 
   the we\-ak supersolution and the weak solution of problem \eqref{eq:Problem}$_\lambda$
   obtained in Lemma \ref{lem:2.3Haitao}, and assume that $0<\lambda<\Lambda_\e$.
   Then, $u_\lambda$ is a \emph{local minimizer} of $I_{\lambda,\e}$
   in \eqref{eq:functionalIlambdae}.
\end{lemma}
\begin{proof}
 By contradiction, suppose
 that $u_\lambda$ \emph{is not} a local minimizer for $I_{\lambda,\e}$. Then,
 we can construct a sequence $\{u_j\}_j\subseteq\mathcal{X}^{1,2}(\Omega)$
 satisfying the following properties:
 \vspace*{0.1cm}
 
 i)\,\,$u_j\to u_\lambda$ in $\mathcal{X}^{1,2}(\Omega)$ as $j\to+\infty$;
 \vspace*{0.05cm}
 
 ii)\,\,$I_{\lambda,\e}(u_j) < I_{\lambda,\e}(u_\lambda)$ for every $j\in\mathbb{N}$.
 \vspace*{0.1cm}
 
 \noindent We explicitly observe that, by possibly replacing $u_j$ with $z_j = |u_j|$, we may assume that
 $u_j\geq 0$ a.e.\,in $\Omega$ for every $j\geq 1$. 
 In fact, since $u_j\to u_\lambda$ in $\mathcal{X}^{1,2}(\Omega)$
 and since $u_\lambda > 0$ almost everywhere in $\Omega$, it is easy to recognize that
 $$\text{$|u_j|\to |u_\lambda| = u_\lambda$ in $\mathcal{X}^{1,2}(\Omega)$ as $j\to+\infty$},$$
 and this shows that property i) is still satisfied by $\{z_j\}_j$. Moreover, we have
 $$I_{\lambda,\e}(|u_j|) \leq I_{\lambda,\e}(u_j) < I_{\lambda,\e}(u_\lambda)\quad\text{for every $j\geq 1$},$$
 and this shows that also property ii) is still satisfied by the sequence $\{z_j\}_j$.
 Hence, from now on we tacitly understand that $\{u_j\}_j$ is a sequence of \emph{non-negative functions}
 satisfying properties i)-ii) above.
  Accordingly, we set 
  $$v_j := 
 \max\{\underline{u},\min\{\overline{u},u_j\}\}\in\mathcal{X}^{1,2}(\Omega)$$ 
 and we define
 \begin{itemize}
  \item[$(\ast)$] $\overline{w}_j = (u_j-\overline{u})_+\in\mathcal{X}^{1,2}_+(\Omega)$ 
  and $\overline{S}_j = \mathrm{supp}(\overline{w}_j)
  = \{u_j\geq \overline{u}\}$;
    \vspace*{0.1cm}
    
   \item[$(\ast)$] $\underline{w}_j = (u_j-\underline{u})_-\in\mathcal{X}^{1,2}_+(\Omega)$ 
  and $\underline{S}_j = \mathrm{supp}(\underline{w}_j)
  = \{u_j\leq \underline{u}\}$.
 \end{itemize}
 We explicitly observe that, by definition, the following identities hold:
 \begin{equation} \label{eq:indetitiesujvjwj}
  \begin{split}
  \mathrm{a)}&\,\,v_j\in M = \{u\in\mathcal{X}^{1,2}(\Omega):\,\underline{u}\leq u\leq\overline{u}\}; \\
  \mathrm{b)}&\,\,\text{$v_j \equiv \overline{u}$ on $\overline{S}_j$,
  $v_j \equiv \underline{u}$ on $\underline{S}_j$ and $v_j \equiv u_j$ on 
  $\{\underline{u}<u_j<\overline{u}\}$}; \\
  \mathrm{c)}&\,\,\text{$u_j = \overline{u}+\overline{w}_j$ on $\overline{S}_j$ and
  $u_j = \underline{u}-\underline{w}_j$ on $\underline{S}_j$}.
  \end{split}
 \end{equation}
 Following \cite{Haitao}, we now claim that
 \begin{equation} \label{eq:claimmeasureSn}
  \lim_{n\to+\infty}|\overline{S}_j| = \lim_{n\to+\infty}|\underline{S}_j| = 0.
 \end{equation}
 Indeed, let $\sigma > 0$ be arbitrarily fixed and let $\delta > 0$ be such that
 $|\Omega\setminus\Omega_\delta| < \frac{\sigma}{2}$, where we have set
 $\Omega_\delta = \{x\in\Omega:\,d(x,\de\Omega) > \delta\}\Subset\Omega$.
 Since, by construction, 
 $$\underline{u} = w_\lambda\in\mathcal{X}^{1,2}(\Omega)$$ is the unique solution
 of problem \eqref{eq:NoPerturbation}, 
 we know from Proposition \ref{prop:esistenzaGarain} that 
 \begin{equation} \label{eq:underustaccata}
  \text{$u_\lambda\geq \underline{u}\geq C > 0$ a.e.\,in $\Omega_\delta$},
  \end{equation}
 where $C = C(\delta,\underline{u}) > 0$ is a suitable constant 
 (recall that $\underline{u}\leq u_\lambda\leq \overline{u}$).
 \vspace*{0.05cm}
 
 On the other hand, since $u_\lambda$ is a weak solution of problem \eqref{eq:Problem}$_\lambda$,
 and since
 $\overline{u} = u_{\lambda'}$ for some $\lambda<\lambda'<\Lambda$
 (see the proof of Lemma \ref{lem:2.3Haitao}),
 by \eqref{eq:underustaccata} we have 
 \begin{align*}
  \LL (\overline{u}-u_\lambda) & = \lambda'\overline{u}^{-\gamma}-\lambda u_\lambda^{-\gamma}
  +(\overline{u}^{2^*-1}-u_\lambda^{2^*-1}) \\
  & (\text{since, by construction, $\underline{u}\leq u_\lambda\leq \overline{u}$
  and $\lambda < \lambda'$}) \\
  & \geq \lambda(\overline{u}^{-\gamma}-u_\lambda^{-\gamma}) \\
  & (\text{by the Mean Value Theorem, for some $\theta\in(0,1)$}) \\
  & = -\gamma\lambda(\theta\overline{u}+(1-\theta)u_\lambda)^{-\gamma-1}(\overline{u}-u_\lambda) \\
  & \geq -\gamma\lambda \underline{u}^{-\gamma-1}(\overline{u}-u_\lambda) \\
  & (\text{here we use \eqref{eq:underustaccata}}) \\
  & \geq -\gamma\lambda C^{-\gamma-1}(\overline{u}-u_\lambda),
 \end{align*}
 in the weak sense on $\Omega_\delta$; as a consequence, we see that $v := \overline{u}-u_\lambda
 \in\mathcal{X}^{1,2}(\Omega)$ is a \emph{weak supersolution} (in the sense of
 Definition \ref{def:weaksubsolZeroOrder}) of equation \eqref{eq:withZeroOrder}, with
 $$a(x) = \gamma\lambda C(\delta,\overline{u})^{-\gamma-1} > 0.$$
 Since $v > 0$ a.e.\,on every ball $B\Subset\Omega$ (as $\overline{u} = u_{\lambda'}$ and 
  $\lambda\neq\lambda'$), we can apply again Co\-rol\-la\-ry 
  \ref{cor:BrezisNirenbergpernoi},
  ensuring the existence of $C_1 = C_1(\delta,u_\lambda,\overline{u}) > 0$ such that
  \begin{equation} \label{eq:vstaccatafinal}
   \text{$v = \overline{u}-u_\lambda \geq C_1 > 0$ a.e.\,in $\Omega_\delta$}.
  \end{equation}
  With \eqref{eq:vstaccatafinal} at hand, we can finally complete the proof
  of \eqref{eq:claimmeasureSn}: in fact, recalling that
   $u_j\to u_\lambda$ in $\mathcal{X}^{1,2}(\Omega)\hookrightarrow L^2(\Omega)$
  as $j\to+\infty$, from \eqref{eq:vstaccatafinal} we obtain
  \begin{align*}
   |\overline{S}_j| & \leq |\Omega\setminus\Omega_\delta|+|\Omega_\delta\cap\overline{S}_j|
   < \frac{\sigma}{2}+\frac{1}{C_1^2}\int_{\Omega_\delta\cap\overline{S}_j}(\overline{u}-u_\lambda)^2\,dx\\
   & (\text{since $0\leq \overline{u}-u_\lambda\leq u_j-u_\lambda$ a.e.\,in $\overline{S}_j$}) \\
   & < \frac{\sigma}{2}+\frac{1}{C_1^2}\|u_j-u_\lambda\|^2_{L^2(\Omega)} < \sigma,
  \end{align*}
  provided that $j$ is large enough, and this proves that $|\overline{S}_j|\to 0$ as $j\to+\infty$.
  In a very similar fashion, one can prove that $|\underline{S}_j|\to 0$ as $j\to+\infty$.
  \vspace*{0.1cm}
  
  Now we have established \eqref{eq:claimmeasureSn}, we can proceed toward the end of the demonstration
  of the lemma.
  To begin with, using identities b)-c) in \eqref{eq:indetitiesujvjwj} we write
  \begin{align*}
   I_{\lambda,\e}(u_j) & = I_{\lambda,\e}(v_j)
   + \frac{1}{2}\big(\rho_\e(u_j)^2-\rho_\e(v_j)^2\big) \\
   &\qquad-\frac{\lambda}{1-\gamma}\int_\Omega(|u_j|^{1-\gamma}-|v_j|^{1-\gamma})\,dx 
   -\frac{1}{2^*}\int_\Omega(|u_j|^{2^*}-|v_j|^{2^*})\,dx \\
   & = I_{\lambda,\e}(v_j)
   + \frac{1}{2}\int_{\overline{S}_j\cup\underline{S}_j}(|\nabla u_j|^2-|\nabla v_j|^2)\,dx
   +  \frac{\e}{2}\big([u_j]_s^2-[v_j]^2_s\big) \\
   &\qquad-\frac{\lambda}{1-\gamma}\int_{\overline{S}_j\cup\underline{S}_j}
   (|u_j|^{1-\gamma}-|v_j|^{1-\gamma})\,dx 
   -\frac{1}{2^*}\int_{\overline{S}_j\cup\underline{S}_j}(|u_j|^{2^*}-|v_j|^{2^*})\,dx \\
   & = I_{\lambda,\e}(v_j) + \frac{\e}{2}\big([u_j]_s^2-[v_j]^2_s\big)
   +
   \mathcal{R}^{(1)}_j+\mathcal{R}^{(2)}_j = (\bigstar), \phantom{+\int_{\underline{S}_j\cup\overline{S}_j}}
  \end{align*}
  where we have introduced the shorthand notation
  \begin{align*}
   (\ast)&\,\,\mathcal{R}^{(1)}_j = \frac{1}{2}\int_{\overline{S}_j}
   \big(|\nabla (\overline{u}+\overline{w}_j)|^2-|\nabla \overline{u}|^2\big)\,dx \\
   &\qquad - \int_{\overline{S}_j}
  \Big\{\frac{\lambda}{1-\gamma}(|\overline{u}+\overline{w}_j|^{1-\gamma}-|\overline{u}|^{1-\gamma})
  +\frac{1}{2^*}(|\overline{u}+\overline{w}_j|^{2^*}-|\overline{u}|^{2^*})\Big\}dx; \\[0.15cm]
  (\ast)&\,\,\mathcal{R}^{(2)}_j = \frac{1}{2}\int_{\underline{S}_j}
  \big(|\nabla (\underline{u}-\underline{w}_j)|^2-|\nabla \underline{u}|^2\big)\,dx \\
  & \qquad - \int_{\underline{S}_j}
  \Big\{\frac{\lambda}{1-\gamma}(|\underline{u}-\underline{w}_j|^{1-\gamma}-|\underline{u}|^{1-\gamma})
  +\frac{1}{2^*}(|\underline{u}-\underline{w}_j|^{2^*}-|\underline{u}|^{2^*})\Big\}dx.
  \end{align*}
  Then, we exploit the estimate for the \emph{purely nonlocal} term $N_0 = [u_j]_s^2-[v_j]^2_s$ 
  e\-sta\-bli\-shed
  in \cite[Lemma 3.3]{GiacMukSre}: this gives the following computation
  \begin{align*}
   (\bigstar) & \geq I_{\lambda,\e}(v_j)+ \frac{\e}{2}\big([\overline{w}_j]^2_s+
   [\underline{w}_j]^2_s\big)+
   \e\,\langle \overline{u},\overline{w}_j\rangle_{s,\R^{2n}}-
   \e,\langle \underline{u},\underline{w}_j\rangle_{s,\R^{2n}}+
   \mathcal{R}^{(1)}_j+\mathcal{R}^{(2)}_j \\[0.2cm]
   & =
   I_{\lambda,\e}(v_j)+\frac{1}{2}\Big(\int_{\Omega}|\nabla \overline{w}_j|^2\,dx
   + \e\,[\overline{w}_j]^2_s\Big)
   + \frac{1}{2}
   \Big(\int_{\Omega}|\nabla \underline{w}_j|^2\,dx
   + \e\,[\underline{w}_j]^2_s\Big) \\
   & \qquad 
   + \int_{\Omega}\nabla \overline{u}\cdot\nabla \overline{w}_j\,dx 
   + \e\,\langle \overline{u},\overline{w}_j\rangle_{s,\R^{2n}}
   - \int_{\Omega}\nabla \underline{u}\cdot\nabla \underline{w}_j\,dx 
   - \e\,\langle \underline{u},\underline{w}_j\rangle_{s,\R^{2n}} \\
   & \qquad 
   - \int_{\overline{S}_j}
  \Big\{\frac{\lambda}{1-\gamma}(|\overline{u}+\overline{w}_j|^{1-\gamma}-|\overline{u}|^{1-\gamma})
  +\frac{1}{2^*}(|\overline{u}+\overline{w}_j|^{2^*}-|\overline{u}|^{2^*})\Big\}dx \\
  & \qquad 
  -\int_{\underline{S}_j}
  \Big\{\frac{\lambda}{1-\gamma}(|\underline{u}-\underline{w}_j|^{1-\gamma}-|\underline{u}|^{1-\gamma})
  +\frac{1}{2^*}(|\underline{u}-\underline{w}_j|^{2^*}-|\underline{u}|^{2^*})\Big\}dx \\
   & = I_{\lambda,\e}(v_j)+\frac{1}{2}\rho_\e(\overline{w}_j)^2
   +\frac{1}{2}\rho_\e(\underline{w}_j)^2
   +\mathcal{B}_\e(\overline{u},\overline{w}_j)-
   \mathcal{B}_\e(\underline{u},\underline{w}_j) \\
   & \qquad - \int_{\overline{S}_j}
  \Big\{\frac{\lambda}{1-\gamma}(|\overline{u}+\overline{w}_j|^{1-\gamma}-|\overline{u}|^{1-\gamma})
  +\frac{1}{2^*}(|\overline{u}+\overline{w}_j|^{2^*}-|\overline{u}|^{2^*})\Big\}dx \\
  & \qquad 
  -\int_{\underline{S}_j}
  \Big\{\frac{\lambda}{1-\gamma}(|\underline{u}-\underline{w}_j|^{1-\gamma}-|\underline{u}|^{1-\gamma})
  +\frac{1}{2^*}(|\underline{u}-\underline{w}_j|^{2^*}-|\underline{u}|^{2^*})\Big\}dx.
  \end{align*}
  Summing up, we obtain
  $$I_{\lambda,\e}(u_j) = I_{\lambda,\e}(v_j)+A_j+B_j,$$
  where we have used the notation
 \begin{align*}
  (\ast)&\,\,A_j = \frac{1}{2}\rho_\e(\overline{w}_j)^2+
  \mathcal{B}_\e(\overline{u},\overline{w}_j) 
  \\
  & \qquad - \int_{\overline{S}_j}
  \Big\{\frac{\lambda}{1-\gamma}(|\overline{u}+\overline{w}_j|^{1-\gamma}-|\overline{u}|^{1-\gamma})
  +\frac{1}{2^*}(|\overline{u}+\overline{w}_j|^{2^*}-|\overline{u}|^{2^*})\Big\}dx;
  \\[0.15cm]
 (\ast)&\,\,B_j  = \frac{1}{2}\rho_\e(\underline{w}_j)^2-\mathcal{B}_\e(\underline{u},\underline{w}_j)
  \\
  & \qquad - \int_{\underline{S}_j}
  \Big\{\frac{\lambda}{1-\gamma}(|\underline{u}-\underline{w}_j|^{1-\gamma}-|\underline{u}|^{1-\gamma})
  +\frac{1}{2^*}(|\underline{u}-\underline{w}_j|^{2^*}-|\underline{u}|^{2^*})\Big\}dx.
 \end{align*}
 Now, since we have already recognized that
 $v_j\in M$
 and since, by con\-stru\-ction, we know that $I_{\lambda,\e}(u_\lambda) = \inf_M I_{\lambda,\e}$ 
 (see the proof of Lemma \ref{lem:2.3Haitao}), we get
 \begin{equation} \label{eq:dovecontraddire}
   I_{\lambda,\e}(u_j)\geq I_{\lambda,\e}(u_\lambda)+A_j+B_j.
  \end{equation}
 On the other hand, since $\overline{u}=u_{\lambda'}$ is a \emph{weak supersolution}
 of \eqref{eq:Problem}$_\lambda$, we have
 \begin{align*}
  A_j & = \frac{1}{2}\rho_\e(\overline{w}_j)^2
  + \mathcal{B}_{\e}(\overline{u},\overline{w}_j)
  \\
  & \qquad - \int_{\overline{S}_j}
  \Big\{\frac{\lambda}{1-\gamma}(|\overline{u}+\overline{w}_j|^{1-\gamma}-|\overline{u}|^{1-\gamma})
  +\frac{1}{2^*}(|\overline{u}+\overline{w}_j|^{2^*}-|\overline{u}|^{2^*})\Big\}dx \\
  & (\text{by the Mean Value Theorem, for some $\theta\in (0,1)$}) \\
  & \geq \frac{1}{2}\rho_\e(\overline{w}_j)^2
  + \int_{\overline{S}_j}(\lambda \overline{u}^{-\gamma}+\overline{u}^{2^*-1})\overline{w}_j\,dx
  \\
  &\qquad -\int_{\overline{S}_j}\big\{\lambda(\overline{u}+\theta\overline{w}_j)^{-\gamma}
  \overline{w}_j+(\overline{u}+\theta\overline{w}_j)^{2^*-1}\overline{w}_j\}dx \\
  & \geq  \frac{1}{2}\rho_\e(\overline{w}_j)^2-\int_{\overline{S}_j}
  \big((\overline{u}+\theta\overline{w}_j)^{2^*-1}-\overline{u}^{2^*-1}\big)\overline{w}_j\,dx
  \\
  & (\text{again by the Mean Value Theorem}) \\
  & \geq \frac{1}{2}\rho_\e(\overline{w}_j)^2-C\int_{\overline{S}_j}(\overline{u}^{2^*-2}+\overline{w}_j^{2^*-2})
  \overline{w}_j^2\,dx,
 \end{align*}
 where $C > 0$ is a suitable constant only depending on the dimension $n$.
 From this, by exploiting H\"older's and Sobolev's inequalities
 (see \eqref{eq:equivalenceuniforme}), we obtain
 \begin{equation} \label{eq:todeduceAngeqzero}
  \begin{split}
  A_j & = \frac{1}{2}\rho_\e(\overline{w}_j)^2-C\int_{\overline{S}_j}(\overline{u}^{2^*-2}+\overline{w}_j^{2^*-2})
  \overline{w}_j^2\,dx \\
  & \geq 
  \frac{1}{2}\rho_\e(\overline{w}_j)^2\Big\{1-\hat{C}\Big(\int_{\overline{S}_j}\overline{u}^{2^*}\,dx
   \Big)^{\frac{2^*-2}{2^*}}-\hat{C}\rho_\e(\overline{w}_j)^{2^*-2}\Big\},
  \end{split}
 \end{equation}
 where $\hat{C} > 0$ is another constant depending on $n$.
 
 With \eqref{eq:todeduceAngeqzero} at hand, we are finally ready to complete the proof.
 Indeed, taking into account the above \eqref{eq:claimmeasureSn}, we have
 \begin{align*}
  \lim_{n\to+\infty}\Big(\int_{\overline{S}_j}\overline{u}^{2^*}\,dx
   \Big)^{\frac{2^*-2}{2^*}} = 0; 
   \end{align*}
   moreover, since $u_j\to u_\lambda$ in $\mathcal{X}^{1,2}(\Omega)$ as $j\to+\infty$, one also get
   \begin{align*}
   0& \leq \rho_\e(\overline{w}_j)^2
   \leq \Theta\|\overline{w}_j\|_{H^1_0(\Omega)}^2 = \Theta\int_{\overline{S}_j}|\nabla (u_j-
    \overline{u})|^2\,dx\\
    & \leq 2\Theta\|u_j-u_\lambda\|_{H^1_0(\Omega)}^2+2\Theta\int_{\overline{S}_j}
    |\nabla (u_\lambda-\overline{u})|^2\,dx \\
    & \leq 2\Theta\rho_\e(u_j-u_\lambda)^2+2\Theta\int_{\overline{S}_j}
    |\nabla (u_\lambda-\overline{u})|^2\,dx\to 0\qquad\text{as $j\to+\infty$},
 \end{align*}
 where we have also used the equivalence between $\rho_\e$ and $\|\cdot\|_{H^1_0(\Omega)}$, see
 \eqref{eq:equivalenceuniforme}.
 Ga\-the\-ring these facts, we then infer the existence of some $j_0\geq 1$ such that
 $$A_j\geq \frac{1}{2}\rho_\e(\overline{w}_j)^2\Big\{1-\hat{C}\Big(\int_{\overline{S}_j}\overline{u}^{2^*}\,dx
   \Big)^{\frac{2^*-2}{2^*}}-\hat{C}\rho(\overline{w}_j)^{2^*-2}\Big\} \geq 0\quad
   \text{for all $j\geq j_0$}.$$
   By arguing in a very similar way, one can prove that $B_j\geq 0$ for every $j\geq j_0$
   (by possibly enlarging $j_0$ if needed); as a consequence, from 
   \eqref{eq:dovecontraddire} we get
   $$I_{\lambda,\e}(u_j)\geq I_{\lambda,\e}(u_\lambda) +A_j+B_j\geq I_{\lambda,\e}(u_\lambda),$$
   but this is contradiction with property ii) of the sequence $\{u_j\}_j$.
\end{proof}
By combining all the results established so far, we can give the
\begin{proof}[Proof (of Theorem \ref{thm:main}).] 
    Let $\Lambda_\e$ be as in \eqref{eq:DefinitionLambda}, that is, 
    $$\Lambda_\e := \sup \{ \lambda >0: \eqref{eq:Problem}_\lambda \textrm{ admits a weak solution}\}.$$
    On account of Lemma \ref{lem:Lambdafinito}, we know that $\Lambda_\e\in (0,+\infty)$; moreover,
    by Lemma \ref{lem:2.3Haitao} we know that
    problem 
    \eqref{eq:Problem}$_\Lambda$ possesses at least one weak solution $u_\lambda$
    for every $0<\lambda\leq\Lambda_\e$,
    and thus assertion a) in the statement of the theorem is established.     
    
    On the other hand,
    by the very
    definition of $\Lambda_\e$ we derive that \eqref{eq:Problem}$_\lambda$
    \emph{does not admit} weak solutions when $\lambda > \Lambda_\e$, and this establishes also
    assertion b).
    
    Finally, if $0<\lambda<\Lambda_\e$
    we know from Lemma \ref{lem:25Haitao} that the solution $u_\lambda$ is
    a local mi\-ni\-mi\-zer of $I_{\lambda,\e}$, and the proof is complete.
    \end{proof}
\section{Existence of a second solution}\label{sec:proofThm2}
As described in the Introduction, in this last section
we provide the proof
of the multiplicity result in Theorem \ref{thm:main2}; since this result
holds for $\e$ sufficiently small (hence, when $\LL$ is sufficiently
close to $-\Delta$), here we need to \emph{move $\e$}
and to exploit the \emph{uniform estimates} proved in the previous sections.
\medskip

To begin with, we fix once and for all a number $r > 0$ in such a way that
\emph{both}
 Lemma \ref{lem:Lambdafinito}\,-\,i) and Theorem \ref{thm:uniformLinf}\,-\,ii)
 do apply; then, we can find $\lambda_* > 0$, only depending on $r_0,\gamma$ and on the measure
 of $\Omega$, such that the following facts hold:
\begin{itemize}
 \item[a)]
 \emph{for every $0<\e\leq 1$ and for every $0<\lambda<\lambda_*$},
 problem \eqref{eq:Problem}$_\lambda$ possesses a weak solution $
 u = u_{\lambda,\e}\in\mathcal{X}^{1,2}(\Omega)$, further satisfying
 \begin{itemize}
  \item[{i)}] $\|u_{\lambda,\e}\|_{H_0^1(\Omega)}\leq r$;
  \item[{ii)}] $u_{\lambda,\e}$ is a local minimizer of the functional $I_{\lambda,\e}$
  in \eqref{eq:functionalIlambdae}.
\end{itemize}
 \item[b)] there exists a constant $C = C(n,\lambda_*,\Omega) > 0$ such that
 $$\|u_{\lambda,\e}\|_{L^{\infty}(\Omega)} \leq C.$$
\end{itemize}
We explicitly point out that, even if it is not obtained
by the Perron method in Lemma \ref{lem:2.2Haitao}, the function $u_{\lambda,\e}$ 
\emph{is a weak solution
of problem \eqref{eq:Problem}$_\lambda$}; as a consequence, from the proof
of Lemma \ref{lem:2.3Haitao} we still derive that
\begin{equation} \label{eq:lowerboundWMP}
 u_{\lambda,\e}\geq w_{\lambda,\e}\quad\text{a.e.\,in $\Omega$},
\end{equation}
where $w_{\lambda,\e}$ is the unique solution of the purely singular
problem \eqref{eq:NoPerturbation}.
\vspace*{0.1cm}

We now turn to prove the following key lemma.
\begin{lemma}\label{lem:CrucialLemma}
Let $\lambda \in (0,\lambda_{\star})$ be fixed \emph{(}with $\lambda_*>0$ as in assertion 2\emph{)}\emph{)}.
Then, there exist $\varepsilon_{0} \in (0,1)$, $R_{0}>0$ and a positive function $\Psi\in\mathcal{X}^{1,2}(\Omega)$ such that
\begin{equation} \label{eq:CrucialLemmaEq}
\left\{ \begin{array}{lr}
I_{\lambda}(u_{\lambda}+R\Psi) < I_{\lambda}(u_{\lambda}) & \textrm{for all  $\varepsilon \in (0,\varepsilon_{0})$ and $R\geq R_{0}$},\\
I_{\lambda}(u_{\lambda}+t R_{0}\Psi) < I_{\lambda}(u_{\lambda}) + \tfrac{1}{n} S_{n}^{n/2} & \textrm{for all  $\varepsilon \in (0,\varepsilon_{0})$ and $t \in [0,1]$}.
\end{array}\right.
\end{equation}
\begin{proof}
  We closely follow the approach in \cite[Lemma 2.7]{Haitao}. First of all, we 
  choose a \emph{Lebesgue point} of $u_\lambda$ in $\Omega$, say $y$, 
  and we let $r > 0$ be such that $B_r(y)\Subset\Omega$;
  we then choose a cut-off function $\varphi\in C_0^\infty(\Omega)$ such that 
  \begin{itemize}
   \item[$(\ast)$] $0\leq\varphi\leq 1$ in $\Omega$;
   \item[$(\ast)$] $\varphi\equiv 1$ on $B_r(y)$;
  \end{itemize}
  and we consider the one-parameter family of functions
  \begin{equation}\label{eq:Scelta_Talenti}
  U_\e = V_\e\,\varphi,\qquad
  \text{where}\,\,V_\e(x) = \frac{\e^{\frac{\alpha(n-2)}{2}}}{(\e^{2\alpha}+|x-y|^2)^{\frac{n-2}{2}}},
  \end{equation}
  where $\alpha > 0$ will be appropriately chosen later on.
  Notice that this family $\{V_\e\}_\e$
  is the well-kno\-wn family of the \emph{Aubin-Ta\-len\-ti functions}, which are
  the unique (up to translation) extremals in the (local) Sobolev inequality; 
  this means that
  \begin{equation} \label{eq:Veminim}
   \frac{\||\nabla V_\e|\|^2_{L^2(\R^n)}}{\|V_\e\|^2_{L^{2^*}(\R^n)}} = S_n,
  \end{equation}
  where $S_n > 0$ is the \emph{best constant} in the Sobolev inequality.   
  
  \noindent We now recall that, owing to \cite[Lemma 1.1]{BN}, we have (as $\e \to 0^+$)
  \begin{equation} \label{eq:estimVeLocalBN}
   \begin{split}
    \mathrm{i)}&\,\,\| U_\e\|^2_{H^1_0(\Omega)} = 
    \e^{\alpha(n-2)}\Big(\frac{K_1}{\e^{\alpha(n-2)}}+O(1)\Big)
    = K_1+o(\e^{\frac{\alpha(n-2)}{2}}); \\[0.1cm]
    \mathrm{ii)}&\,\,\|U_\e\|^{2^*}_{L^{2^*}(\Omega)} = \e^{\alpha n}\Big(\frac{K_2}{\e^{\alpha n}}+O(1)\Big)
    = K_2+o(\e^{\alpha n/2});
   \end{split}
  \end{equation}
  where the constants $K_1,K_2 > 0$ are given, respectively, by
  $$K_1 = \||\nabla V_1|\|^2_{L^2(\R^n)},\qquad K_2 = \int_{\R^n}\frac{1}{(1+|x|^2)^n}\,dx$$
  and they satisfy 
  $K_1/K_2^{1-2/n} = S_n$, see the above \eqref{eq:Veminim}.
  On the other hand, by combining
   \eqref{eq:estimVeLocalBN}\,-\,ii) with
   \cite[Lemma 4.11]{BDVV5},
  we also have 
  \begin{equation} \label{eq:estimVeNonlocal}
  \begin{split}
   \e \, [U_\e]^2_s  & \leq \e \, \|U_\e\|^2_{L^{2^*}(\Omega)}\cdot O(\e^{\alpha(2-2s)}) = \|U_\e\|^2_{L^{2^*}(\Omega)}\cdot O(\e^{1+\alpha(2-2s)})
   \\
   & = \big(K_2+o(\e^{\alpha n/2})\big)^{2/2^*}\cdot O(\e^{1+\alpha(2-2s)}) = o(\e^{1+\alpha(1-s)}).
   \end{split}
  \end{equation}
Defining
  $w = u_\lambda + tRU_\e$ (for $R\geq 1$ and $t\in[0,1]$),
  we then obtain 
  \begin{equation*}
  \begin{split}
   I_{\lambda,\e}(w) & = I_{\lambda,\e}(u_\lambda) 
   + \frac{t^2R^2}{2}\rho_\e(U_\e)^2+tR\mathcal{B}_{\rho_{\e}}(u_\lambda,U_\e)
   \\
   & \qquad
   - \frac{1}{2^*}\big(\|u_\lambda+tRU_\e\|^{2^*}_{L^{2^*}(\Omega)}-
    \|u_\lambda\|^{2^*}_{L^{2^*}(\Omega)}\big) \\
   &\qquad
   -\frac{\lambda}{1-\gamma}\Big(\int_{\Omega}|u_\lambda+tRU_\e|^{1-\gamma}\,dx
   - \int_{\Omega}|u_\lambda|^{1-\gamma}\,dx\Big) ,
   \end{split}
   \end{equation*}
   \noindent and recalling that $u_\lambda$ solves \eqref{eq:Problem}$_{\lambda}$, we further get
    \begin{equation*}
  \begin{split}
   I_{\lambda,\e}(w) &= I_{\lambda,\e}(u_\lambda)
   + \frac{t^2R^2}{2}\rho_\e(U_\e)^2
   +tR\int_\Omega(\lambda u_\lambda^{-\gamma}+u_\lambda^{2^*-1})U_\e\,d x
   \\
   & \qquad
   - \frac{1}{2^*}\big(\|u_\lambda+tRU_\e\|^{2^*}_{L^{2^*}(\Omega)}-
    \|u_\lambda\|^{2^*}_{L^{2^*}(\Omega)}\big) \\
   &\qquad
   -\frac{\lambda}{1-\gamma}\Big(\int_\Omega|u_\lambda+tRU_\e|^{1-\gamma}\,dx
   - \int_\Omega|u_\lambda|^{1-\gamma}\,dx\Big) \\
   & = I_{\lambda,\e}(u_\lambda)
   + \frac{t^2R^2}{2}\rho_\e(U_\e)^2
   - \frac{t^{2^*}R^{2^*}}{2^*}\|U_\e\|^{2^*}_{L^{2^*}(\Omega)} \\
   & \qquad - t^{2^*-1}R^{2^*-1}\int_\Omega U_\e^{2^*-1}u_\lambda\,dx
   - \mathcal{R}_\e-\mathcal{D}_\e, \\
  \end{split}
  \end{equation*}
   where we have introduced the notation
   \begin{align*}
    (\ast)&\,\,\mathcal{R}_\e = \frac{1}{2^*}\int_\Omega\Big\{|u_\lambda+tRU_\e|^{2^*}
    - u_\lambda^{2^*}-(tRU_\e)^{2^*}
    \\
    & \qquad\qquad\qquad 
     - 2^* u_\lambda (tRU_\e)\big(u_\lambda^{2^*-2}+ (tRU_\e)^{2^*-2}\big)\Big\}dx; \\
     (\ast)&\,\,\mathcal{D}_\e = 
     \frac{\lambda}{1-\gamma}\int_\Omega
     \Big\{|u_\lambda+tRU_\e|^{1-\gamma}- |u_\lambda|^{1-\gamma}-
     tR(1-\gamma)u_\lambda^{-\gamma}U_\e\Big\}dx.
   \end{align*}

We start estimating $\mathcal{R}_{\e}$. To this aim, we follow \cite[Proof of Theorem 1]{BN89} (from equation (17) on) where the main difference is due to the fact that $u_{\lambda}$ is actually dependent on $\e$ as well. However, by either using Proposition \ref{prop:esistenzaGarain}-iii) (recalling that the solution of the purely singular problem is a sub-solution of \eqref{eq:Problem}$_{\lambda}$), or the uniform upper bound \eqref{eq:Stima_indip_da_epsilon} of the $L^{\infty}$-norm of $u_{\lambda}$ one can get
      \begin{equation} \label{eq:termRe}
    |-\mathcal{R}_\e| \leq R^\beta o(\e^{\frac{\alpha(n-2)}{2}})\quad\text{as $\e\to 0^+$},
   \end{equation}
   for some $\beta\in(0,2^*)$.
Regarding $\mathcal{D}_{\e}$, we can argue as in \cite[Lemma 2.7]{Haitao}, once again keeping in mind that now $u_{\lambda}$ depends on $\e$ as well. However, thanks to 
\eqref{eq:lowerboundWMP}
and to Proposition \ref{prop:esistenzaGarain}-iii), we can exploit \cite[equation (5.7)]{KRS}, getting
    \begin{equation} \label{eq:termDe}
    -\mathcal{D}_\e \leq C(tR+t^2R^2)\,o(\e^{\frac{\alpha(n-2)}{2}})\quad\text{as $\e\to 0^+$},
   \end{equation}
   for some constant $C > 0$ independent of $\e$.
   Finally, combining once again
   \eqref{eq:lowerboundWMP} and Proposition \ref{prop:esistenzaGarain}-iii), and arguing as in
    \cite[Proof of Lemma 4.11]{BDVV5}, we get
   \begin{equation}\label{eq:termUeTarantello}
   \begin{aligned}
     \int_\Omega U_\e^{2^*-1}u_\lambda\,dx &\geq 
     \int_{B_{r}(y)}V_{\e}^{2^*-1}w_{\lambda,\e}\,dx \geq c_2 \, \int_{B_{r}(y)}V_{\e}^{2^*-1}\,dx \\
     &\geq
     \e^{\alpha(n-2)/2}D_0 +
     o(\e^{\alpha(n-2)/2})\quad\text{as $\e\to 0^+$},
     \end{aligned}
   \end{equation}
 where $D_0 > 0$ is a constant only depending on $n$.

    Gathering \eqref{eq:estimVeLocalBN}-\eqref{eq:termUeTarantello}, we finally obtain
   \begin{equation}
   \begin{aligned}
   & I_{\lambda,\e} \leq I_{\lambda,\e}(u_\lambda)+\frac{t^2R^2}{2}K_1-
   \frac{t^{2^*}R^{2^*}}{2^*}K_2 \\
   &\qquad - t^{2^*-1}R^{2^*-1} D_0 \e^{\alpha(n-2)/2}
   - R^{\beta} o(\e^{\alpha(n-2)/2})+ C(tR+t^2R^2)\,o(\e^{\frac{\alpha(n-2)}{2}})
   \\
   &\qquad\qquad +\big(t^2R^2+t^{2^*}R^{2^*}\big)\big(o(\e^{1+\alpha(1-s)})+o(\e^{\alpha(n-2)/2})\big).
   \end{aligned}
\end{equation}    
    Thus, if we choose $\alpha\in(0,1]$ so small that  
   $1+\alpha(1-s)>\alpha(n-2)/2$
   (notice that this is always possible, since $n\geq 3$ and $s\in(0,1)$), we get (as $\e \to 0^+$)
\begin{equation} \label{eq:toconcludecomeinTar}
   \begin{split}
   I_\lambda(u_\lambda+tRU_\e) & \leq 
   I_\lambda(u_\lambda)+\frac{t^2R^2}{2}K_1-
   \frac{t^{2^*}R^{2^*}}{2^*}K_2 \\
   &\qquad - t^{2^*-1}R^{2^*-1}D_0 \e^{\alpha(n-2)/2}
   \\
   &\qquad\qquad 
   +C(tR+t^2R^2+t^{2^*}R^{2^*}+R^\beta)o(\e^{\alpha(n-2)/2}).
   \end{split}
   \end{equation}   
   Notice that the above choice of $\alpha$ is doable thanks to the presence of $\e$ in front of the nonlocal part, thus the above argument cannot be run for $\e=1$.\\
   Now, thanks to estimate \eqref{eq:toconcludecomeinTar} we are finally
   ready to complete the proof of the lemma: in fact, starting from this estimate
   and repeating word by word the argument in \cite[Lemma 3.1]{Tarantello}, we find
   $\e_0 > 0$ and $R_0 > 0$ such that
   \begin{equation*}
    \begin{cases}
     I_\lambda(u_\lambda+RU_\e) < I_\lambda(u_\lambda) & \text{for all $\e\in(0,\e_0)$ and $R\geq R_0$}, \\
      I_\lambda(u_\lambda+tR_0 U_\e) < I_\lambda(u_\lambda)+\frac{1}{n}S_n^{n/2}
      & \text{for all $\e\in(0,\e_0)$ and $t\in[0,1]$}.
    \end{cases}
   \end{equation*}
   Thus, the lemma is proved by choosing $\Psi = U_\e$ (with $\e < \e_0$ and $y,a$ as above).
\end{proof}
\end{lemma}

%

Thanks to Lemma \ref{lem:CrucialLemma}, we can now proceed toward the proof
of Theorem \ref{thm:main2}: in fact, we turn to show that problem
\eqref{eq:Problem}$_{\lambda}$ possesses a second solution $v_\lambda\neq u_\lambda$, provided
that $0<\lambda<\lambda_*$ and $0<\e<\e_0$, \emph{where $\e_0 > 0$ is as in Lemma \ref{lem:CrucialLemma}}.

Let then $0<\e<\e_0$ be arbitrarily but fixed (with $\e_0 > 0$ as in Lemma
\ref{lem:CrucialLemma}), and let $0<\lambda <\lambda_*$. Since we know that $u_\lambda$
is a local minimizer of the functional $I_\lambda$,
there exists some $0<\varrho_0 = \varrho_0(\e) \leq \rho_\e(u_\lambda)$ such that
 \begin{equation} \label{eq:ulambdaminEkeland}
  I_\lambda(u)\geq I_\lambda(u_\lambda)\quad\text{for every $u\in\mathcal{X}^{1,2}(\Omega)$ with
 $\rho_\e(u-u_\lambda)<\varrho_0$}.
 \end{equation}
 As a consequence of \eqref{eq:ulambdaminEkeland}, if we consider the cone 
 \begin{equation*}
  T = \{u\in\mathcal{X}^{1,2}(\Omega):\,\text{$u\geq u_\lambda>0$ a.e.\,in $\Omega$}\},
 \end{equation*}
 only one of the following two cases hold:
 \begin{itemize}
  \item[\textsc{A)}] $\inf\{I_\lambda(u):\,\text{$u\in T$ and $\rho_\e(u-u_\lambda) = \varrho$}\}
  = I_\lambda(u_\lambda)$ for every $0<\varrho<\varrho_0$;
  \vspace*{0.05cm}
  \item[\textsc{B)}] there exists $\varrho_1\in(0,\varrho_0)$ such that 
  \begin{equation} \label{eq:defvarrho}
   \inf\{I_\lambda(u):\,\text{$u\in T$ and $\rho_\e(u-u_\lambda) = \varrho_1$}\} > I_\lambda(u_\lambda).
   \end{equation}
 \end{itemize}
 We explicitly notice that use of $\rho_\e(\cdot)$ to define an open neighborhood
 of $u_\lambda$ in \eqref{eq:ulambdaminEkeland} is motivated by the fact that, since $\e\in (0,1)$,
 by Remark \ref{rem:spaceX12}\,-\,2) we have
 $$\|u\|_{H^1_0(\Omega)}\leq\rho_\e(u)\leq \rho(u)\leq \Theta\|u\|_{H^1_0(\Omega)}$$
 for some $\Theta > 0$ only depending on $n,s$, see \eqref{eq:equivalencerhoH01}; as a consequence,
 the norm $\rho_\e(\cdot)$ is globally equivalent to $\rho(\cdot)$ (and to the 
 $H_0^1$-norm), uniformly w.r.t.\,$\e$.
 \medskip

 Following \cite{Haitao}, we then turn to consider the cases A)-B) separately.
 \begin{proposition} \label{prop:Lemma26Haitao}
  Assume that 
  \textsc{Case A)} holds. Then, for e\-ve\-ry $\varrho\in(0,\varrho_0)$ there exists
  a solution $v_\lambda$ of problem $(\mathrm{P}_{\lambda,\e})$ such that
  $$\text{$\rho_\e(u_\lambda-v_\lambda) = \varrho$}.$$
  In particular, $v_\lambda\not\equiv u_\lambda$.
 \end{proposition}
 \begin{proof}
  Let $0<\varrho<\varrho_0$ be arbitrarily fixed. Since we are assuming that \textsc{Case A)} holds,
  we can find a sequence $\{u_k\}_k\subseteq T$ satisfying the following properties:
  \begin{itemize}
   \item[a)] $\rho_\e(u_k-u_\lambda) = \varrho$ for every $k\geq 1$;
   \vspace*{0.05cm}
   
   \item[b)] $I_\lambda(u_k)\to I_\lambda(u_\lambda) =:\mathbf{c}_\lambda$ as $k\to+\infty$.
  \end{itemize}
  We then choose $\delta > 0$ so small that $\varrho-\delta > 0$ and $\varrho+\delta<\varrho_0$ and,
  accordingly, we consider the subset of $T$ defined as follows:
  $$X = \{u\in T:\,\varrho-\delta\leq\rho_\e(u-u_\lambda)\leq \varrho+\delta\}
  \subseteq T$$
  (note that $u_k\in X$ for every $k\geq 1$, see a)). Since it is \emph{closed}, this set $X$
  is a \emph{complete metric space} when endowed with the distance induced by $\rho$;
  moreover, since $I_\lambda$ is a \emph{real-valued and continuous functional} on $X$, and since
  $$\textstyle\inf_{X}I_\lambda = I_\lambda(u_\lambda)$$
  we are entitled to apply
  the Ekeland Variational Principle (see \cite{AubEke}) to the functional $I_\lambda$ on $X$,
  providing us with a sequence $\{v_k\}_k\subseteq X$ such that
  \begin{equation} \label{eq:EkelandCaseA}
  \begin{split}
    \mathrm{i)}&\,\,I_\lambda(v_k)\leq I_\lambda(u_k)\leq I_\lambda(u_\lambda)+1/k^2, \\
    \mathrm{ii)}&\,\,\rho_\e(v_k-u_k)\leq 1/k, \\
    \mathrm{iii)}&\,\,I_\lambda(v_k)\leq I_\lambda(u)+1/k\cdot
     \rho_{\e}(v_k-u)\quad\text{for every $u\in X$}.    
    \end{split}
  \end{equation}
  We now observe that, since $\{v_k\}_k\subseteq X$ and since the set $X$ is \emph{bounded}
 in $\mathcal{X}^{1,2}(\Omega)$, 
  there exists  $v_\lambda\in\mathcal{X}^{1,2}(\Omega)$ such that (as $k\to+\infty$ and 
  up to a sub-sequence)
  \begin{equation} \label{eq:limitvkCaseA}
  \begin{split}
   \mathrm{i)}&\,\,\text{$v_k\to v_\lambda$ weakly in $\mathcal{X}^{1,2}(\Omega)$}; \\
   \mathrm{ii)}&\,\,\text{$v_k\to v_\lambda$ strongly in $L^p$ for every $1\leq p<2^*$}; \\
   \mathrm{iii)}&\,\,\text{$v_k\to v_\lambda$ pointwise
   a.e.\,in $\Omega$}.
   \end{split}
  \end{equation}
  where we have also used the \emph{compact embedding} 
  $\mathcal{X}^{1,2}(\Omega)\hookrightarrow L^2(\Omega)$. 
  To complete the proof, we then turn to prove the following two facts:
  \begin{itemize}
   \item[1)] $v_\lambda$ is a weak solution of \eqref{eq:Problem}$_{\lambda}$;
   \item[2)] $\rho_\e(v_\lambda-u_\lambda) = \varrho > 0$.
  \end{itemize}
  \vspace*{0.1cm}
  
  \noindent \emph{Proof of} 1). To begin with, we fix
  $w\in T$ and we choose $\nu_0 = \nu_0(w,\lambda) > 0$ so small that 
   $u_\nu = v_k+\nu(w-v_k) \in X$ for every $0<\nu<\nu_0$.
   We explicitly stress that the existence of such an $\nu_0$ easily follows
   from \eqref{eq:EkelandCaseA}-ii) and property a) above.
   \vspace*{0.05cm}
   
   On account of \eqref{eq:EkelandCaseA}-iii)
   (with $u = u_\nu$), we have
   $$\frac{I_\lambda(v_k+\nu(w-v_k))-I_\lambda(v_k)}{\nu}
   \geq -\frac{1}{k}\rho_\e(w-v_k);$$
   From this, by letting $\nu\to 0^+$ and by proceeding \emph{exactly}
   as in the proof of \cite[Lemma 2.2]{Haitao}, we obtain the following estimate
   \begin{equation} \label{eq:221Haitao}
   \begin{split}
    -\frac{1}{k}\rho_\e(w-v_k) & \leq \mathcal{B}_{\e}(v_k,w-v_k)
    - \int_\Omega v_k^{2^*-1}(w-v_k)\,dx \\
    & \qquad-\lambda\int_\Omega v_k^{-\gamma}(w-v_k)\,dx,
   \end{split}
   \end{equation}
   Now, given any $\varphi\in\mathcal{X}^{1,2}(\Omega)$ and any $\nu > 0$, we define
   \begin{itemize}
    \item[$(\ast)$] $\psi_{k,\nu} = v_k+\nu\varphi-u_\lambda$ and $\phi_{k,\nu} = (\psi_{k,\nu})_-$;
    \vspace*{0.05cm}
    
    \item[$(\ast)$] $\psi_\nu = v_\lambda+\nu\varphi-u_\lambda$ and $\phi_\nu = (\psi_\nu)_-$.
   \end{itemize}
   Since, obviously, $w = v_k+\nu\varphi+\phi_{k,\nu}\in T$,
   by exploiting \eqref{eq:221Haitao} we get
   \begin{equation} \label{eq:topasstothelimitCaseASol}
    \begin{split}
	-\frac{1}{k}\rho_\e(\nu\varphi+\phi_{k,\nu}) & \leq \mathcal{B}_{\e}(v_k,\nu\varphi+\phi_{k,\nu})
    - \int_\Omega v_k^{2^*-1}(\nu\varphi+\phi_{k,\nu})\,dx \\
    & \qquad-\lambda\int_\Omega v_k^{-\gamma}(\nu\varphi+\phi_{k,\nu})\,dx.
    \end{split}
   \end{equation}
   Then, we aim to pass to the limit as $k\to+\infty$ and $\nu\to 0^+$ in the above 
   \eqref{eq:topasstothelimitCaseASol}. To this end we first observe that, 
   on account of \eqref{eq:limitvkCaseA}-iii), we have
   \begin{equation} \label{eq:phiketophik}
    \text{$\phi_{k,\nu}\to\phi_\nu$ pointwise a.e.\,in $\Omega$ as $k\to+\infty$};
   \end{equation}
   moreover, by the very definition of $\phi_{k,\nu}$ we also have the following estimate
   $$|\phi_{k,\nu}| =  (u_\lambda-\nu\varphi-v_k)\cdot
    \mathbf{1}_{\{u_\lambda-\nu\varphi-v_k
    \geq 0\}}\leq u_\lambda+\nu|\varphi|;$$
    thus, since $v_k\geq u_\lambda > 0$ a.e.\,in $\Omega$ (as $v_k\in X\subseteq T$), we get
   \begin{equation} \label{eq:perfareLebesgueCaseA}
   \begin{split}
     \mathrm{i)}&\,\,v_k^{2^*-1}|\phi_{k,\nu}|
    = v_k^{2^*-1}(u_\lambda-\nu\varphi-v_k)\cdot
    \mathbf{1}_{\{u_\lambda-\nu\varphi-v_k
    \geq 0\}}\leq (u_\lambda+\nu|\varphi|)^{2^*};\\
    \mathrm{ii)}&\,\,v_k^{-\gamma}|\nu\varphi+\phi_{k,\nu}|
    \leq u_\lambda^{-\gamma}(u_\lambda+2\nu|\varphi|).
    \end{split}
   \end{equation}
    On account \eqref{eq:phiketophik}-\eqref{eq:perfareLebesgueCaseA},
    and since $u_\lambda^{-\gamma}|\varphi|\in L^1(\Omega)$
    (see Remark \ref{rem:defweaksolPb}-2)), we can apply
    Lebesgue's Dominated Convergence theorem, obtaining
    \begin{equation*}
      \begin{split}
       & \lim_{k\to+\infty}
        \int_\Omega v_k^{2^*-1}\phi_{k,\nu}\,dx
        = \int_\Omega v_\lambda^{2^*-1}\phi_\nu\,dx, \\
        & \lim_{k\to+\infty}
        \int_\Omega v_k^{-\gamma}(\nu\varphi+\phi_{k,\nu})\,dx
        = \int_\Omega v_\lambda^{-\gamma}(\nu\varphi+\phi_\nu)\,dx.
      \end{split}
    \end{equation*}
   This, together with \eqref{eq:limitvkCaseA}-ii), allows us to conclude that
   \begin{equation} \label{eq:limitNONLIN}
   \begin{split}
    & \lim_{k\to+\infty}
    \Big(\int_\Omega v_k^{2^*-1}(\nu\varphi+\phi_{k,\nu})\,dx 
    +\lambda\int_\Omega v_k^{-\gamma}(\nu\varphi+\phi_{k,\nu})\,dx\Big) \\
    & \qquad\qquad
    = \int_\Omega v_\lambda^{2^*-1}(\nu\varphi+\phi_{\nu})\,dx 
    + \lambda\int_\Omega v_\lambda^{-\gamma}(\nu\varphi+\phi_{\nu})\,dx.
    \end{split}
   \end{equation}
   As regards the \emph{operator term} $\mathcal{B}_{\e}(v_k,\nu\varphi+\phi_{k,\nu})$
   we first observe that, by using the computations
   already carried out in 
   \cite[Lemma 3.4]{BadTar} (for the local part) and in \cite[Lemma 4.1]{GiacMukSre}
   (for the nonlocal part), we have
   \begin{equation*}
    \begin{split}
    & \mathcal{B}_{\e}(v_k,\phi_{k,\nu})  \\
    &\qquad = \int_{\Omega}\nabla v_k\cdot\nabla \phi_{k,\nu}\,dx
      + \e\iint_{\R^{2n}}\frac{(v_k(x)-v_k(y))(\phi_{k,\nu}(x)-\phi_{k,\nu}(y))}{|x-y|^{n+2s}}\,dx\,dy
      \\
      & \qquad =
      \int_{\{v_k+\nu\varphi\leq u_\lambda\}}
      \nabla v_k\cdot\nabla (u_\lambda-\nu\varphi-v_\lambda)\,dx  \\
      & \qquad\qquad +  \int_{\{v_k+\nu\varphi\leq u_\lambda\}}
      \nabla v_k\cdot\nabla (v_\lambda-v_k)\,dx \\
      & \qquad\qquad 
       + \e\iint_{\R^{2n}}\frac{(v_k(x)-v_k(y))(\phi_{\nu}(x)-\phi_{\nu}(y))}{|x-y|^{n+2s}}\,dx\,dy \\
      & \qquad\qquad
      + \e\iint_{\R^{2n}}\frac{(v_k(x)-v_k(y))}{|x-y|^{n+2s}}
      \big((\phi_{k,\nu}-\phi_\nu)(x)-(\phi_{k,\nu}-\phi_\nu)(y))\big)\,dx\,dy\\
      & \qquad \leq \int_\Omega \nabla v_k\cdot\nabla \phi_\nu\,dx \\
      &\qquad\qquad + 
      \e\iint_{\R^{2n}}\frac{(v_k(x)-v_k(y))(\phi_{\nu}(x)-\phi_{\nu}(y))}{|x-y|^{n+2s}}\,dx\,dy
      +o(1) \\
      & \qquad = 
      \mathcal{B}_{\e}(v_k,\phi_\nu)+o(1)\qquad\text{as $k\to+\infty$}.
      \phantom{\iint_{\R^{2n}}}
    \end{split}
   \end{equation*}
   This, together with the fact that $v_k\to v_\lambda$ weakly in $\mathcal{X}^{1,2}(\Omega)$,
   gives
   \begin{equation} \label{eq:limitOPERATORPART}
    \mathcal{B}_{\e}(v_k,\nu\varphi+\phi_{k,\nu}) \leq \mathcal{B}_{\e}(v_\lambda,\nu\varphi+\phi_{\nu})+
    o(1)\qquad\text{as $k\to+\infty$}.
   \end{equation}
   Gathering \eqref{eq:limitNONLIN} and \eqref{eq:limitOPERATORPART}, and taking into account
   that ${\rho_\e}(\phi_{k,\nu})$ is \emph{uniformly bo\-un\-ded} with respect to $k$ 
   (as the same is true of $v_k$), we can finally
   pass to the limit as $k\to+\infty$ in \eqref{eq:topasstothelimitCaseASol}, obtaining
   \begin{equation} \label{eq:afterlimitkCaseA}
    \begin{split}
     \mathcal{B}_{\e}(v_\lambda,\nu\varphi+\phi_{\nu})
     \geq \int_\Omega v_\lambda^{2^*-1}(\nu\varphi+\phi_{\nu})\,dx 
    + \lambda\int_\Omega v_\lambda^{-\gamma}(\nu\varphi+\phi_{\nu})\,dx.
    \end{split}
   \end{equation}
   With \eqref{eq:afterlimitkCaseA} at hand, we can now exploit
   once again the computations carried out in \cite[Lemma 2.6]{Haitao}
   and in \cite[Lemma 4.1]{GiacMukSre}, getting
   \begin{equation} \label{eq:226Haitao}
    \begin{split}
    & \mathcal{B}_{\e}(v_\lambda,\varphi)-
     \lambda\int_\Omega v_\lambda^{-\gamma}\varphi\,dx 
     - \int_\Omega v_\lambda^{2^*-1}\varphi\,dx  \\
     & \qquad
     \geq-\frac{1}{\nu}\Big(\mathcal{B}_{\e}(v_\lambda,\phi_\nu)
     - \lambda\int_\Omega v_\lambda^{-\gamma}\phi_\nu\,dx 
     - \int_\Omega v_\lambda^{2^*-1}\phi_\nu\,dx\Big) \\
     & \qquad = \frac{1}{\nu}\Big(-\mathcal{B}_{\e}(v_\lambda-u_\lambda,\phi_\nu)
     + \lambda\int_\Omega (v_\lambda^{-\gamma}-u_\lambda^{-\gamma})\phi_\nu\,dx \\
     & \qquad\qquad +
     \int_\Omega (v_\lambda^{2^*-1}-u_\lambda^{2^*-1})\phi_\nu\,dx\Big) = (\bigstar),
    \end{split}
   \end{equation}
   where we have used the fact that $u_\lambda$ is a solution 
   of problem \eqref{eq:Problem}$_\lambda$;
   from this, recalling that
   $v_\lambda = \textstyle\lim_{k\to+\infty}v_k\geq u_\lambda$ 
   (see the above \eqref{eq:limitvkCaseA}\,-\,iii)), we obtain
   \begin{align*}
    (\bigstar) & \geq \frac{1}{\nu}\Big(
    -\int_{\{v_\lambda+\nu\varphi\leq u_\lambda\}}\nabla (v_\lambda-u_\lambda)\cdot
    \nabla (v_\lambda-u_\lambda+\nu\varphi)\,dx \\
    & \qquad\qquad-\e\iint_{\R^{2n}}\frac{((v_\lambda-u_\lambda)(x)-(v_\lambda-u_\lambda)(y))
    (\phi_\nu(x)-\phi_\nu(y))}{|x-y|^{n+2s}}\,dx\,dy\\
    & \qquad\qquad+ \lambda\int_{\{v_\lambda+\nu\varphi\leq u_\lambda\}} (v_\lambda^{-\gamma}-u_\lambda^{-\gamma})
    (v_\lambda-u_\lambda-\nu\varphi)\,dx\Big) \\
    &\qquad \geq \text{$o(1)$ as $\nu\to 0^+$}; \phantom{\iint_{\R^{2n}}}
   \end{align*}
   of problem 
   as a consequence, by letting $\nu\to 0^+$ in \eqref{eq:226Haitao}, we obtain
   $$\mathcal{B}_{\e}(v_\lambda,\varphi)-
     \lambda\int_\Omega v_\lambda^{-\gamma}\varphi\,dx 
     - \int_\Omega v_\lambda^{2^*-1}\varphi\,dx\geq 0.
     $$
     This, together with the \emph{arbitrariness} of the fixed $\varphi\in\mathcal{X}^{1,2}(\Omega)$,
     finally proves that the function
     $v_\lambda$ is a weak solution of problem \eqref{eq:Problem}$_{\lambda}$,
     as claimed.
     \medskip
     
     \noindent \emph{Proof of} 2). To prove assertion 2) it suffices to show that
     \begin{equation} \label{eq:claimvkstrongconv}
      \text{$v_k\to v_\lambda$ strongly in $\mathcal{X}^{1,2}(\Omega)$ as $k\to+\infty$}.
     \end{equation}
     In fact, owing to
     property a) of $\{u_k\}_k$ we have
     $$\varrho-{\rho_\e}(u_k-v_k)\leq{\rho_\e}(v_k-u_\lambda) \leq {\rho_\e}(v_k-u_k)+\varrho;$$
     this, together with \eqref{eq:claimvkstrongconv} and 
     \eqref{eq:EkelandCaseA}-ii), ensures that ${\rho_\e}(u_\lambda-v_\lambda) = \varrho.$ Hence,
     we turn to to prove \eqref{eq:claimvkstrongconv}, namely
      the \emph{strong convergence} of $\{v_k\}_k$ to $v_\lambda$. 
     \vspace*{0.1cm}
     
     First of all, since $v_k\to v_\lambda$ weakly in $\mathcal{X}^{1,2}(\Omega)$
     as $k\to+\infty$,
     we can proceed as in the proof
     of Lemma \ref{lem:Lambdafinito}, obtaining the following analogs
     of \eqref{eq:2.3Haitao}-to-\eqref{eq:2.5Haitao}:
     \begin{equation} \label{eq:comeLemma21Haitao}
     \begin{split}
      (\ast)&\,\,
       \int_{\Omega}v_{k}^{1-\gamma} \, dx = \int_{\Omega}v_\lambda^{1-\gamma}\, dx + o(1), \\
       (\ast)&\,\,\|v_k\|^{2^*}_{L^{2^*}(\Omega)} = \|v_\lambda\|^{2^*}_{L^{2^*}(\Omega)} + 
        \|v_k - v_\lambda\|^{2^*}_{L^{2^*}(\Omega)} +o(1);\phantom{\int_\Omega} \\
        (\ast)&\,\,{\rho_\e}(v_k)^2 = {\rho_\e}(v_\lambda)^2 + {\rho_\e}(v_k -v_\lambda)^2 + o(1).
        \phantom{\int_\Omega}
        \end{split}
     \end{equation}
     Moreover, by \eqref{eq:limitvkCaseA}-ii), we also get
     \begin{equation} \label{eq:comeHaitaoBis}
      \lim_{k\to+\infty}\int_\Omega|v_k-v_\lambda|^{1-\gamma}\,dx = 0
     \end{equation}
     Owing to \eqref{eq:comeLemma21Haitao}, 
     and choosing $w = v_\lambda\in T$ in \eqref{eq:221Haitao}, we then get
     \begin{equation*}
   \begin{split}
    & {\rho_\e}(v_k-v_\lambda)^2 
   = -\mathcal{B}_{\e}(v_k,v_\lambda-v_k)+\mathcal{B}_{\e}(v_\lambda,v_\lambda-v_k) \\
   & \qquad\leq \frac{1}{k}{\rho_\e}(v_k-v_\lambda)
    + \lambda\int_\Omega v_k^{-\gamma}(v_k-v_\lambda)\,dx
    \\
    & \qquad\qquad + \int_\Omega v_k^{2^*-1}(v_k-v_\lambda)\,dx
    + \mathcal{B}_{\e}(v_\lambda,v_\lambda-v_k) \\
    & \qquad (\text{since $\{v_k\}_k$ is bounded and $v_k\to v_\lambda$ weakly in $\mathcal{X}^{1,2}(\Omega)$}) \\
    & \qquad= \lambda\int_\Omega v_k^{1-\gamma}\,dx +
    \int_\Omega v_k^{2^*-1}(v_k-v_\lambda)\,dx - 
    \lambda\int_\Omega v_k^{-\gamma}v_\lambda\,dx + o(1) \\
    & \qquad
    = \lambda\int_\Omega v_\lambda^{1-\gamma}\,dx +
        \|v_k - v_\lambda\|^{2^*}_{L^{2^*}(\Omega)}
        +\|v_\lambda\|^{2^*}_{L^{2^*}(\Omega)}-\int_\Omega v_k^{2^*-1}v_\lambda\,dx
    \\
    &\qquad\qquad - 
    \lambda\int_\Omega v_k^{-\gamma}v_\lambda\,dx + o(1) \\
    & \qquad (\text{since $v_k\to v_\lambda$ strongly in $L^p(\Omega)$ for all $1\leq p< 2^*$}) \\
    & \qquad = \lambda\int_\Omega v_\lambda^{1-\gamma}\,dx +
    \|v_k - v_\lambda\|^{2^*}_{L^{2^*}(\Omega)}
    - 
    \lambda\int_\Omega v_k^{-\gamma}v_\lambda\,dx + o(1).
   \end{split}
   \end{equation*}
   On the other hand, since $0\leq v_k^{-\gamma}v_\lambda\leq u_\lambda^{-\gamma}v_\lambda\in L^1(\Omega)$
   (recall that $v_\lambda\geq u_\lambda$ and see
   Remark \ref{rem:defweaksolPb}-ii)), by Lebesgue's Dominated Convergence Theorem we have
   $$\int_\Omega v_k^{-\gamma}v_{\lambda}\,dx \to \int_\Omega v_\lambda^{1-\gamma}\,dx\qquad\text{
   as $k\to+\infty$};$$
   as a consequence, we obtain
   \begin{equation} \label{eq:228Haitao}
     {\rho_\e}(v_k-v_\lambda)^2\leq \|v_k-v_\lambda\|^{2^*}_{L^{2^*}(\Omega)}+o(1)\quad\text{as $k\to+\infty$}.
   \end{equation}
   To proceed further, we now choose $w = 2v_k\in T$ in \eqref{eq:221Haitao}: this yields
   \begin{equation*}
    {\rho_\e}(v_k)^2-\|v_k\|^{2^*}_{L^{2^*}(\Omega)}-\lambda\int_\Omega v_k^{1-\gamma}\,dx\geq -
    \frac{1}{k}{\rho_\e}(v_k)^2 = o(1);
   \end{equation*}
   thus, recalling that $v_\lambda$ is a weak solution of problem $(\mathrm{P}_{\lambda,\e})$,
   we get
   \begin{equation} \label{eq:231Haitao}
    \begin{split}
     {\rho_\e}(v_k-v_\lambda)^2 & =
     {\rho_\e}(v_k)^2-{\rho_\e}(v_\lambda)^2+o(1)  \\
     & \geq \Big(\|v_k\|^{2^*}_{L^{2^*}(\Omega)}+\lambda\int_\Omega v_k^{1-\gamma}\,dx\Big)
     - \mathcal{B}_{\e}(v_\lambda,v_\lambda) \\
     & = \|v_k\|^{2^*}_{L^{2^*}(\Omega)}+\lambda\int_\Omega v_k^{1-\gamma}\,dx
     - \|v_\lambda\|^{2^*}_{L^{2^*}(\Omega)}-\lambda\int_\Omega v_\lambda^{1-\gamma}\,dx \\
     & = \|v_k-v_\lambda\|^{2^*}_{L^{2^*}(\Omega)}+o(1)\quad\text{as $k\to+\infty$},
    \end{split}
   \end{equation}
   where we have also used \eqref{eq:comeLemma21Haitao}. Gathering 
   \eqref{eq:228Haitao}-\eqref{eq:231Haitao}, we then obtain
   \begin{equation} \label{eq:232Haitao}
    {\rho_\e}(v_k-v_\lambda)^2 = \|v_k-v_\lambda\|^{2^*}_{L^{2^*}(\Omega)}+o(1)\quad\text{as $k\to+\infty$}.
   \end{equation}
   With \eqref{eq:232Haitao} at hand, we can finally end the proof
   of \eqref{eq:claimvkstrongconv}. In fact, assuming (to fix the ideas) that $I_\lambda(u_\lambda)
   \leq I_\lambda(v_\lambda)$, from \eqref{eq:EkelandCaseA} and 
   \eqref{eq:comeLemma21Haitao}-\eqref{eq:comeHaitaoBis} we get
   \begin{align*}
    I_\lambda(v_k-v_\lambda) & = \frac{1}{2}{\rho_\e}(v_k-v_\lambda)^2
    + \frac{\lambda}{1-\gamma}\int_\Omega|v_k-v_\lambda|^{1-\gamma}\,dx
    + \frac{1}{2^*}\|v_k-v_\lambda\|^{2^*}_{L^{2^*}(\Omega)} \\
    & = I(v_k)-I(v_\lambda)+o(1) \leq I(u_\lambda)-I(v_\lambda)+\frac{1}{k^2}+o(1) \\
    & = o(1)\qquad\text{as $k\to+\infty$};
   \end{align*}
   this, together with \eqref{eq:comeHaitaoBis}, gives
   \begin{equation} \label{eq:233Haitao}
   \begin{split}
    & \frac{1}{2}{\rho_\e}(v_k-v_\lambda)^2
    -\frac{1}{2^*}\|v_k-v_\lambda\|^{2^*}_{L^{2^*}(\Omega)} \\
    & \qquad
     = I_\lambda(v_k-v_\lambda)+ \frac{\lambda}{1-\gamma}\int_\Omega|v_k-v_\lambda|^{1-\gamma}\,dx
     \leq o(1).
    \end{split}
   \end{equation}
   Thus, by combining \eqref{eq:232Haitao}-\eqref{eq:233Haitao}, we easily obtain
   $$\lim_{k\to+\infty}\|v_k-v_\lambda\|^{2^*}_{L^{2^*}(\Omega)} = \lim_{k\to+\infty}
   {\rho_\e}(v_k-v_\lambda)^2  =0,$$
   and this proves \eqref{eq:claimvkstrongconv}.
 \end{proof}
 \begin{proposition} \label{prop:Lemma27Haitao}
  Assume that 
  \textsc{Case B)} holds. Then, there exists
  a second solution $v_\lambda$ of problem \eqref{eq:Problem}$_{\lambda}$ such that
  $v_\lambda\not\equiv u_\lambda$.
 \end{proposition}
 \begin{proof}
  To begin with, we consider the set
  $$\Gamma = \big\{\eta\in C([0,1];T):\,\text{$\eta(0) = u_\lambda,\,I_\lambda(\eta(1))<I_\lambda(u_\lambda)$
  and ${\rho_\e}(\eta(1)-u_\lambda) > \varrho_1$}\big\},$$
  (where $\varrho_1 > 0$ is as in \eqref{eq:defvarrho}), and we claim that $\Gamma\neq\varnothing$.
  
  In fact, since the fixed $\e$ satisfies $0<\e<\e_0$ (with $\e_0 > 0$ as
  in Lemma \ref{lem:CrucialLemma}), 
  by the cited Lemma \ref{lem:CrucialLemma} we know that there exists $R_0 > 0$ such that
   \begin{equation} \label{eq:claimHaitao}
    \begin{cases}
     I_\lambda(u_\lambda+RU_\e) < I_\lambda(u_\lambda) & \text{for all $R\geq R_0$}, \\
      I_\lambda(u_\lambda+tR_0 U_\e) < I_\lambda(u_\lambda)+\frac{1}{n}S_n^{n/2}
      & \text{for all $t\in[0,1]$}.
    \end{cases}
   \end{equation}
   In particular, from \eqref{eq:claimHaitao} we easily see that
   $$\eta_0(t) = u_\lambda+tR_0U_\e\in \Gamma$$
   (by enlarging $R_0$ if needed),
   and thus $\Gamma\neq \varnothing$, as claimed.
   \vspace*{0.1cm}
   
   Now we have proved that $\Gamma\neq\varnothing$, we can proceed towards the end of the proof. 
   To this end we first observe that, since it is non-empty, this set $\Gamma$ is a
   \emph{complete metric space}, when endowed with the distance
   $$d_\Gamma(\eta_1,\eta_2) := \max_{0\leq t\leq 1}{\rho_\e}\big(\eta_1(t)-\eta_2(t)\big);$$
   moreover, since $I_\lambda$ is \emph{real-valued and continuous} on $\mathcal{X}^{1,2}(\Omega)$,
   it is easy to recognize that the functional $\Phi:\Gamma\to \R$ defined as
   $$\Phi(\eta) := \max_{0\leq t\leq 1}I_\lambda(\eta(t)),$$
   is (well-defined and) continuous on $\Gamma$. In view of these facts,
   we are then entitled to apply the Ekeland Variational Principle
   to this functional $\Phi$ on $\Gamma$: setting
   $$\gamma_0 := \inf{\Gamma}\Phi(\eta),$$
   there exists a sequence $\{\eta_k\}_k\subseteq\Gamma$ such that
    \begin{equation} \label{eq:EkelandCaseB}
  \begin{split}
    \mathrm{i)}&\,\,\Phi(\eta_k)\leq \gamma_0+1/k, \\
    \mathrm{ii)}&\,\,\Phi(\eta_k)\leq \Phi(\eta)+1/k\,d_\Gamma(\eta_k,\eta) \quad\text{for every $\eta\in\Gamma$}.    
    \end{split}
  \end{equation}
  Now, starting from
  \eqref{eq:EkelandCaseB} and proceeding exactly as in the proof of \cite[Lemma 3.5]{BadTar} (see also
  \cite[Lemma 4.3]{GiacMukSre}), we can find another sequence
  $$v_k = \eta_k(t_k)\in T$$
  (for some $t_k\in[0,1]$) such that
  \vspace*{0.1cm}
  
    a)\,\,$I_\lambda(v_k)\to\gamma_0$ as $k\to+\infty$;
    \vspace*{0.05cm}
    
    b)\,\,there exists some $C > 0$ such that, for every $w\in T$, one has
    \begin{equation} \label{eq:237Haitao}
    \begin{split}
     & \mathcal{B}_{\e}(v_k,w-v_k)
      - \lambda\int_\Omega v_k^{-\gamma}(w-v_k)\,dx \\
      &\qquad\qquad 
      -\int_\Omega v_k^{2^*-1}(w-v_k)\,dx \geq -\frac{C}{k}(1+{\rho_\e}(w)).
     \end{split}
    \end{equation}
    In particular, choosing $w = 2v_k$ in \eqref{eq:237Haitao}, we get
    \begin{equation} \label{eq:choice2vk}
     {\rho_\e}(v_k)^2-\lambda\int_\Omega v^{1-\gamma}\,dx
     -\int_\Omega v_k^{2^*}\,dx \geq -\frac{C}{k}(1+2{\rho_\e}(v_k)).
    \end{equation}
    By combining \eqref{eq:choice2vk} 
    with assertion a), and exploiting H\"older's and Sobolev's i\-ne\-qua\-lities,
    we then obtain the following estimate
    \begin{equation} \label{eq:dadedurrevkbounded}
     \begin{split}
     \gamma_0+o(1) & = \frac{1}{2}{\rho_\e}(v_k)^2-
     \frac{\lambda}{1-\gamma}\int_\Omega v_k^{1-\gamma}\,dx
     - \frac{1}{2^*}\int_\Omega v_k^{2^*}\,dx
     \\
     & \geq \Big(\frac{1}{2}-\frac{1}{2^*}\Big){\rho_\e}(v_k)^2
     - \lambda\Big(\frac{1}{1-\gamma}-\frac{1}{2^*}\Big)\int_\Omega v_k^{1-\gamma}\,dx
     \\
     & \qquad -\frac{C}{2^*\,k}(1+2{\rho_\e}(v_k)) \\
     & \geq \Big(\frac{1}{2}-\frac{1}{2^*}\Big)
      {\rho_\e}(v_k)^2- C\big({\rho_\e}(v_k)^{1-\gamma}-2{\rho_\e}(v_k)-1\big),
     \end{split}
    \end{equation}
    where $C > 0$ is a constant depending on $n$ and on $|\Omega|$. Since, obviously,
    $$c_0 = \frac{1}{2}-\frac{1}{2^*} > 0,$$
    it is readily seen from \eqref{eq:dadedurrevkbounded} that the sequence $\{v_k\}_k$
    is \emph{bounded in $\mathcal{X}^{1,2}(\Omega)$} (otherwise, by possibly choosing a sub-sequence
    we would have ${\rho_\e}(v_k)\to+\infty$, and hence
    the right-hand side of \eqref{eq:dadedurrevkbounded} would diverges as $k\to+\infty$, which
    is not possible). 
    
    In view of this fact, we can thus proceed as in the proof of Lemma \ref{prop:Lemma26Haitao} to
    show that $\{v_k\}_k$ weakly converges (up to a sub-sequence)
    to a \emph{weak solution} $v_\lambda\in\mathcal{X}^{1,2}(\Omega)$ of 
    problem \eqref{eq:Problem}$_{\lambda}$,
    further satisfying the identity
    \begin{equation} \label{eq:238Haitao}
     {\rho_\e}(v_k-v_\lambda)^2-\|v_k-v_\lambda\|^{2^*}_{L^{2^*}(\Omega)} = o(1)\quad
     \text{as $k\to+\infty$}.
    \end{equation}
    In view of these facts, to complete the proof we are left to show that
    $v_\lambda\not\equiv u_\lambda$. To this end we first observe that,
    given any $\eta\in\Gamma$, we have
    $$\text{${\rho_\e}(\eta(0)-u_\lambda) = 0$\quad and \quad ${\rho_\e}(\eta(1)-u_\lambda) > \varrho_1$},$$
    and hence there exists a point $t_\eta\in[0,1]$ such that ${\rho_\e}(\eta(t_\eta)-u_\lambda) = \varrho_1$;
    as a con\-se\-quence, since \emph{we are assuming that} \textsc{Case B)} holds, we obtain
    \begin{align*}
     \gamma_0 & = \inf_{\eta\in\Gamma}\Phi(\eta)
     \geq \inf\big\{I_\lambda(\eta(t_\eta)):\,\eta\in\Gamma\big\} \\
     & \geq \inf\{I_\lambda(u):
     \text{$u\in T$ and ${\rho_\e}(u-u_\lambda) = \varrho_1$}\} > I_\lambda(u_\lambda).
    \end{align*}
    On the other hand, since we already know that $\eta_0(t) = u_\lambda+tR_0U_\e\in \Gamma$,
    from \eqref{eq:claimHaitao} (and the very definition of $\gamma_0$) we derive the following
    estimate
    \begin{align*}
     \gamma_0 & \leq \Phi(\eta_0) = \max_{0\leq t\leq 1}I_\lambda(\eta_0(t)) <
     I_\lambda(u_\lambda)+\frac{1}{n}S_n^{n/2}.
    \end{align*}
    Summing up, we have
    \begin{equation} \label{eq:236Haitao}
     I_\lambda(u_\lambda) < \gamma_0 < I_\lambda(u_\lambda)+\frac{1}{n}S_n^{n/2}.
    \end{equation}
    Now, since the sequence $\{v_k\}_k$ weakly converges in $\mathcal{X}^{1,2}(\Omega)$
    to $v_\lambda$ as $k\to+\infty$,
    \emph{the same as\-ser\-tions} in \eqref{eq:comeLemma21Haitao}
    hold also in this context; this, together with \eqref{eq:236Haitao}
    and the above property a) of the sequence $\{v_k\}_k$, gives
    \begin{equation} \label{eq:239Haitao}
     \begin{split}
      & \frac{1}{2}{\rho_\e}(v_k-v_\lambda)^2
    -\frac{1}{2^*}\|v_k-v_\lambda\|^{2^*}_{L^{2^*}(\Omega)} \\
    & \qquad
     =\frac{1}{2}\big({\rho_\e}(v_k)^2-{\rho_\e}(v_\lambda)^2\big)
      -\frac{1}{2^*}\big(\|v_k\|^{2^*}_{L^{2^*}(\Omega)}
      - \|v\|^{2^*}_{L^{2^*}(\Omega)}\big)+o(1)
     \\[0.15cm]
     & \qquad 
     = I_\lambda(v_k)-I_\lambda(u_\lambda)+o(1) 
     = \gamma_0-I_\lambda(u_\lambda)+o(1) \\
     & \qquad < \frac{1}{n}S_n^{n/2}-\delta_0,
     \end{split}
    \end{equation}
    for some $\delta_0 > 0$ such that $1/n\,S_n^{n/2}-\delta_0 > 0$ (provided that $k$ is large enough).
    
   Gathering \eqref{eq:238Haitao}-\eqref{eq:239Haitao}, 
   and arguing as in \cite[Proposition 3.1]{Tarantello}, it is then easy to recognize that
    $v_k\to v_\lambda$ \emph{strongly in $\mathcal{X}^{1,2}(\Omega)$};
    as a consequence, by the continuity of the functional
    $I_\lambda$ and by \eqref{eq:237Haitao}-\eqref{eq:236Haitao}, we get
    $$I_\lambda(u_\lambda) < \gamma_0 =  \lim_{k\to+\infty}I_\lambda(v_k) = I_\lambda(v_\lambda),$$
    and this finally proves that $v_\lambda\not\equiv u_\lambda$, as desired.
    \end{proof}
    Using all the results established so far, we are finally ready to provide the
    \begin{proof}[Proof (of Theorem \ref{thm:main2}).] 
    Let $\e_0 > 0$ be as in Lemma \ref{lem:CrucialLemma}, and let $\lambda_* > 0$ be as in Lemma 
    \ref{lem:Lambdafinito}; moreover, let $0<\e<\e_0$ and $0<\lambda<\lambda_*$ be fixed.
    
    As already observed at the beginning of the section, by scrutinizing the proof
    of Lemma \ref{lem:Lambdafinito} it is easy to recognize that 
    $\Lambda_\e\geq \lambda_*$ (independently of $\e$), 
    where $\Lambda_\e>0$ is the optimal threshold of solvability of
    problem \eqref{eq:Problem}$_{\lambda}$ provided by Theorem \ref{thm:main}; as a consequence,
    by the cited Theorem \ref{thm:main}, we infer that
    there exists a first weak solution $u_\lambda$ of this problem.
    
    On the other hand, by combining Propositions \ref{prop:Lemma26Haitao}-\ref{prop:Lemma27Haitao}
    (which can be applied, since we are assuming that $\e < \e_0$),
    we deduce that there exists another weak solution $v_\lambda$ of problem \eqref{eq:Problem}$_{\lambda}$
    such that $v_\lambda\neq u_\lambda$, and the proof is complete.
    \end{proof}
	\begin{proof}[Proof (of Corollary \ref{cor:BrezisNirenbergpernoi}).] 
	 Assume that $n = 3$ and $0<s<1/2$. By scrutinizing the proof
	 of Lemma \ref{lem:CrucialLemma}, we 
	 see that the cited Lemma \ref{lem:CrucialLemma} holds
	  for $\mathcal{L}_{1} = -\Delta+(-\Delta)^s$: this means, precisely, that
	 there exist $\nu_0 > 0$ and $R_0 > 0$ such that
	 \begin{equation*}
\left\{ \begin{array}{lr}
I_{\lambda}(u_{\lambda}+RU_\nu) < I_{\lambda}(u_{\lambda}) & \textrm{for all  $\nu \in (0,\nu_{0})$ and $R\geq R_{0}$},\\
I_{\lambda}(u_{\lambda}+t R_{0}U_\nu) < I_{\lambda}(u_{\lambda}) + \tfrac{1}{n} S_{n}^{n/2} & \textrm{for all  $\nu \in (0,\nu_{0})$ and $t \in [0,1]$},
\end{array}\right.
\end{equation*}
where $U_\nu$ is as in \eqref{eq:Scelta_Talenti} (with $\e = \nu$ and $\alpha = 1$), and
$$I_\lambda(u) =  \dfrac{1}{2}\mathcal{\rho}(u)^2 - \dfrac{\lambda}{1-\gamma}\int_{\Omega}|u|^{1-\gamma}\, dx - \dfrac{1}{2^{\ast}}\int_{\Omega}|u|^{2^{\ast}}\, dx.$$
With this result at hand, we then see that Propositions \ref{prop:Lemma26Haitao}-\ref{prop:Lemma27Haitao}
hold (with the same proof) for the operator $\LL$, and 
of Corollary \ref{cor:BrezisNirenbergpernoi} now follows by
repeating word by word the demonstration of Theorem \ref{thm:main2} (with $\e = 1$).
\end{proof}
 
\end{document}